\documentclass[preprint]{iacrtrans}

\usepackage[dvipsnames,svgnames,x11names]{xcolor}
\usepackage{caption}
\usepackage{subcaption}
\usepackage{enumitem}
\usepackage{pgfplots}\pgfplotsset{compat=1.18}
\usepackage{orcidlink}
\usepackage{cleveref}

\setlist[enumerate,1]{label={(\arabic*)}}

\allowdisplaybreaks

\theoremstyle{plain}
\newtheorem{thm}{Theorem}[section]
\newtheorem{lem}[thm]{Lemma}
\newtheorem{cor}[thm]{Corollary}
\newtheorem{prop}[thm]{Proposition}

\newtheorem{defn}[thm]{Definition}

\theoremstyle{definition}
\newtheorem{rem}[thm]{Remark}
\newtheorem{ex}[thm]{Example}
\theoremstyle{plain}

\newcommand{\F}{\mathbb{F}}                                         % Field
\newcommand{\Fp}{\F_{p}}                                            % Prime field
\newcommand{\Fq}{\F_{q}}                                            % Prime power field
\newcommand{\Fqn}{\Fq^{n}}                                          % n dimensional vector space over finite field of size q
\newcommand{\Fqm}{\Fq^{m}}                                          % m dimensional vector space over finite field of size q
\newcommand{\Fpx}{\Fp^{\times}}                                     % Invertible elements of finite field of size p
\newcommand{\Fqx}{\Fq^{\times}}                                     % Invertible elements of finite field of size q
\newcommand{\abs}[1]{\left\vert #1 \right\vert}                     % Absoulute value
\newcommand{\C}{\mathbb{C}}                                         % Complex numbers
\DeclareMathOperator{\Frac}{Frac}                                   % Fraction field
\DeclareMathOperator{\id}{id}                                       % Identity
\DeclareMathOperator{\imag}{im}                                     % Image
\newcommand{\legendre}[2]{\left( \frac{#1}{#2} \right)}             % Legendre Symbol
\newcommand{\Q}{\mathbb{Q}}                                         % Rational numbers
\DeclareMathOperator{\Tr}{Tr}                                       % Trace
\newcommand{\Z}{\mathbb{Z}}                                         % Integers
\DeclareMathOperator{\CORR}{CORR}                                   % Correlation from linear cryptanalysis
\newcommand{\degree}[1]{\deg \left( #1 \right)}                     % Degree
\DeclareMathOperator{\wt}{wt}                                       % Hamming weight

\newcommand{\Anemoi}{\texttt{Anemoi}}
\newcommand{\Friday}{\textsc{Friday}}
\newcommand{\Grendel}{\textit{Grendel}}
\newcommand{\Jarvis}{\textsc{Jarvis}}
\newcommand{\Monolith}{\texttt{Monolith}}
\newcommand{\Polocolo}{\textsf{Polocolo}}
\newcommand{\ReinforcedConcrete}{\texttt{Reinforced Concrete}}

\title{A Note on the Walsh Spectrum of Power Residue S-Boxes}

\title{A Note on the Walsh Spectrum of Power Residue S-Boxes}
\author{Matthias Johann Steiner \orcidlink{0000-0001-5206-6579}}
\authorrunning{M.~J.~Steiner}
\institute{Alpen-Adria-Universit\"at Klagenfurt, Klagenfurt am W\"orthersee, Austria \\ \email{mattsteiner@edu.aau.at}}

\begin{document}

    \maketitle

    \keywords[Power residue S-Box, Walsh spectrum, Grendel, Polocolo]{Power residue S-Box, Walsh spectrum, \Grendel, \Polocolo}

    \begin{abstract}
        Let $\mathbb{F}_q$ be a prime field with $q \geq 3$, and let $d, m \geq 1$ be integers such that $\gcd \left( d, q \right) = 1$ and $m \mid (q - 1)$.
        In this paper we bound the absolute values of the Walsh spectrum of S-Boxes $S (x) = x^d \cdot T \left( x^\frac{q - 1}{m} \right)$, where $T$ is a function with $T (x) \neq 0$ if $x \neq 0$.
        Such S-Boxes have been proposed for the Zero-Knowledge-friendly hash functions \textit{Grendel} and \textsc{Polocolo}.
        In particular, we prove the conjectured correlation of the \textsc{Polocolo} S-Box.
    \end{abstract}

    % The content of the paper starts here
    \section{Introduction}
    Let $\Fq$ be a finite field, for $m \in \Z_{> 1}$ such that $m \mid (q - 1)$ we call the map $x \mapsto \legendre{x}{q}_m = x^\frac{q - 1}{m}$ the $m$\textsuperscript{th} power residue, since it maps the elements of $\Fq \setminus \{ 0 \}$ to a subgroup of order $m$.
    For $q$ an odd prime, the canonical example of a power residue is the Legendre symbol
    \begin{equation}
        \legendre{x}{q}_2 =
        \begin{cases}
            0, & x = 0, \\
            1, & x \text{ is a square in } \Fq, \\
            -1, & x \text{ is a non-square in } \Fq.
        \end{cases}
    \end{equation}

    \Grendel{} \cite{EPRINT:Szepieniec21} and \Polocolo{} \cite{EC:HHLPS25} are Zero-Knowledge-friendly (ZK) hash functions defined over prime fields $\Fq$.
    Both are based on the Substitution-Permutation Network.
    As Substitution Box (S-Box) \Grendel{} applies
    \begin{equation}\label{Equ: Grendel S-Box}
        S (x) = x^d \cdot \legendre{x}{q}_2 = x^{d + \frac{q - 1}{2}},
    \end{equation}
    where $d$ is small relative to $q$, and \Polocolo{} applies
    \begin{equation}\label{Equ: power residue S-Box}
        S (x) = x^{q - 2} \cdot T \Bigg( \legendre{x}{q}_m \Bigg),
    \end{equation}
    where $m \in \Z_{> 1}$ is such that $m \mid (q - 1)$ and $T$ is a function on the subgroup of order $m$ of $\Fq \setminus \{ 0 \}$.
    Following the \Polocolo{} designers we call S-Boxes as in \Cref{Equ: Grendel S-Box,Equ: power residue S-Box} \emph{power residue S-Boxes}.

    ZK-friendly hash functions have two novel design criteria compared to their bit-oriented predecessors.
    First, they must be native over a relatively large finite field $\Fq$, where $30 \leq \log_2 \left( q \right) \leq 256$ and typically $q$ is prime.
    Since one of their main applications is the proof of a preimage relation $\mathcal{H} (\mathbf{m}) = \mathbf{h}$, they must also admit efficient ZK-prover circuits.
    To meet this requirement, ZK-friendly hash functions require an efficient arithmetic circuit representation.
    Broken down to the S-Box level, a prover circuit for $S (x) = y$, where $x, y \in \Fq$, must be constructible with a low number of additions and multiplications.
    An obvious choice to meet this criterion are low-degree power permutations $S (x) = x^d$, where $d \in \{ 3, 5, 7 \}$, together with the naive evaluation circuit.
    Another attractive choice is the inverse permutation $S (x) = x^{q - 2}$ of $\Fq$.
    In case that $x \neq 0$, we have the circuit $x \cdot S (x) = x \cdot y = 1$, which requires a single multiplication.
    To cover $x = 0$, we can consider $x^2 \cdot S (x) = x^2 \cdot y = x$, which requires just two multiplications.
    However, if $q$ is large enough, then the probability of $x \neq 0$ (for uniform $x \in \Fq$) is $1 - \frac{1}{q}$, so in practice one can still use $x \cdot y = 1$.
    To the detriment of the inverse S-Box, the arithmetic circuits can also be utilized to set up low degree polynomial models.
    For example, the early ZK-friendly designs \Jarvis{} and \Friday{} \cite{EPRINT:AshDho18} were quickly identified as vulnerable against Gr\"obner basis attacks \cite{AC:ACGKLR19}.
    Since then, the inverse permutation has been discarded for the instantiation of ZK-friendly hash functions, although it enjoys many excellent cryptanalytic properties.
    E.g., a high degree, a low differential uniformity and a high non-linearity.
    \Polocolo{} can be viewed as a fresh attempt to deploy the inverse permutation in ZK hash functions.
    According to the designer's analysis \cite{EC:HHLPS25}, the multiplication with a power residue symbol weakens Gr\"obner basis attacks while maintaining an efficient prover circuit, where the power residue part is proven via a look-up table argument.
    On the other hand, \Grendel's S-Box $S (x) = x^d \cdot \legendre{x}{q}_2$ tries to derive many cryptanalytic properties from $x^d$ while having a larger degree.
    In particular, as prover circuit for $S (x) = y$ one can use $S (x)^2 = x^{2 \cdot d} = y^2$.
    Unfortunately, \Grendel{} has also been demonstrated to be susceptible against Gr\"obner basis attacks \cite{ToSC:GKRS22}.

    Linear cryptanalysis \cite{EC:Matsui93,SAC:BaiSteVau07,AC:Beyne21} is a variant of statistical cryptanalysis which studies linear relations between inputs and outputs of a cryptographic function.
    Resistance of an S-Box $S: \Fq \to \Fq$ is related to the absolute values of the Walsh spectrum of $F$, i.e.\ exponential sums of the form
    \begin{equation}
        \mathcal{W}_S (\psi, X, a, b) = \sum_{x \in X} \psi \big( a \cdot x + b \cdot S (x) \big),
    \end{equation}
    where $X \subset \Fq$,  $\psi: \Fq \to \C$ is a non-trivial additive character and $a, b \in \Fq$.
    (For linear cryptanalysis one usually has $X = \Fq$.)
    While being a standard tool for the analysis of Boolean primitives, the impact of linear cryptanalysis on large prime field primitives is still an open research problem.
    Low degree univariate functions like $S (x) = x^d$ admit an absolute Walsh spectrum of the order $\mathcal{O} \left( q^\frac{1}{2} \right)$ due to Weil's bound \cite{Weil-Bound}.
    However, many ZK-friendly primitives also introduced novel exotic S-Boxes and the estimation of their Walsh spectra poses an intricate problem, see for example the round functions of \Grendel{} \cite{EPRINT:Szepieniec21}, \ReinforcedConcrete{} \cite{CCS:GKLRSW22}, \Anemoi{} \cite{C:BBCPSV23}, \Monolith{} \cite{ToSC:GKLRSW24} and \Polocolo{} \cite{EC:HHLPS25}.
    Henceforth, cryptodesigners, like the \Polocolo{} team, often base their initial linear cryptanalysis on conjectures derived from empirical observations.

    The main research goal of this work is to provide proven bounds on the Walsh spectrum of power residue S-Boxes.
    In particular, we confirm a conjecture of the \Polocolo{} designers \cite[Conjecture~1]{EC:HHLPS25}.
    Hence, our work complements the analysis presented in \cite[\S 5.1]{EC:HHLPS25}, and provides a well-founded argument for the resistance of \Polocolo{} against linear cryptanalysis.

    It is worthwhile to mention that for $X = \Fq$ the Walsh spectrum of the \Polocolo{} S-Box can be expressed in terms of \emph{Kloosterman sums} over a subgroup $\mathcal{G} \subset \Fqx$
    \begin{equation}\label{Equ: Klosterman sum over subgroup}
        K (\psi, \mathcal{G}, a, b) = \sum_{x \in \mathcal{G}} \psi \left( a \cdot x + b \cdot x^{-1} \right),
    \end{equation}
    see \cite[\S 5.1]{EC:HHLPS25} or the proof of \Cref{Cor: Walsh spectrum power residue S-Box inverse}.
    In case that $\mathcal{G} = \Fqx$ and $a, b \in \Fqx$ it is well-known that $\abs{K (\psi, \mathcal{G}, a, b)} \leq 2 \cdot q^\frac{1}{2}$ \cite[5.45.~Theorem]{Lidl-FiniteFields}.
    However, for non-trivial subgroups the \Polocolo{} designers were unable to provide a theoretical bound on the absolute value of \Cref{Equ: Klosterman sum over subgroup}.
    For primes $q < 1000$, $m \in \{ 2 , 4, 8 , 16 \}$ and subgroups $\mathcal{G} \subset \Fqx$ of size $\abs{\mathcal{G}} = \frac{q - 1}{m}$ they empirically observed that $\max_{a, b \in \Fqx} \abs{K (\psi, \mathcal{G}, a, b)} \leq 2 \cdot q^\frac{1}{2}$.
    Based on this observation, the designers conjectured for $q$ prime and $\abs{\mathcal{G}} = \frac{q - 1}{m}$ that \cite[Conjecture~1]{EC:HHLPS25}
    \begin{equation}
        \max_{a, b \in \Fqx} \abs{K (\psi, \mathcal{G}, a, b)} \leq 4 \cdot q^\frac{1}{2},
    \end{equation}
    which was applied for linear cryptanalysis.
    In this note we prove their original empirical observation, namely that for any subgroup $\mathcal{G} \subset \Fqx$ we have
    \begin{equation}
        \max_{a, b \in \Fqx} \abs{K (\psi, \mathcal{G}, a, b)} \leq 2 \cdot q^\frac{1}{2},
    \end{equation}
    and henceforth yield estimations for the full Walsh spectrum  of power residue S-Boxes.
    Our estimation of $\abs{K (\psi, \mathcal{G}, a, b)}$ applies the Weil bound of exponential sums with rational arguments \cite[Theorem~2]{Moreno-Weil-I}.

    Character and Kloosterman sums over $\Fq$ and $\Fqx$ respectively are well-studied objects in the mathematical literature with numerous applications in algebraic geometry, coding theory, cryptography and number theory.
    Though, the special case of a Kloosterman sum over a non-trivial subgroup $\mathcal{G} \subset \Fqx$ seems to be largely absent from the literature.
    One of the rare examples is the work of Ostafe, Shparlinski, and Voloch \cite{Ostafe-Equations}, who estimate \Cref{Equ: Klosterman sum over subgroup} for small subgroups to study the quantum ergodicity of linear maps.

    Additionally, we investigate the Walsh spectrum of power residue S-Boxes, where $x^{q - 2}$ in \Cref{Equ: power residue S-Box} is replaced by some $x^d$ with $d > 1$ and $\gcd \left( d, q \right) = 1$.
    With the classical Weil bound \cite{Weil-Bound} we prove that
    \begin{equation}
        \max_{a, b \in \Fqx} \abs{\mathcal{W}_S (\psi, \Fq, a, b)} \leq (d \cdot m - 1) \cdot q^\frac{1}{2} + 2.
    \end{equation}
    For example, with $m = 2$ this estimation applies to the \Grendel{} S-Box from \Cref{Equ: Grendel S-Box}.

    On the other hand, for power residue S-Boxes with $d = 1$ we prove that
    \begin{equation}
        \max_{a, b \in \Fqx} \abs{\mathcal{W}_S (\psi, \Fq, a, b)} \leq \mathcal{O} \left( \max \left\{ \frac{q - 1}{m}, m \cdot q^\frac{1}{2} \right\} \right).
    \end{equation}
    So, if $m \ll q$, then such S-Boxes admit a high correlation.

    We also investigate a Legendre symbol-based S-Box introduced by Grassi et al.\ \cite[\S 4.1]{ToSC:GKRS22}.
    Its Walsh spectrum will follow as corollary to our estimations for power residue S-Boxes.

    We note that recent Walsh spectrum estimations for cryptographic permutations over prime fields relied on advanced techniques from algebraic geometry \cite{EPRINT:BeyBou24b,Steiner-Flystel}.
    Our estimation of \Cref{Equ: Klosterman sum over subgroup} also falls into this category since the generalized Weil bound \cite[Theorem~2]{Moreno-Weil-I} is proven via the theory of curves over finite fields \cite[\S 5]{Moreno-Weil-I}.

    \section{Preliminaries}
    Let $K$ be a field, we denote with $\bar{K}$ its algebraic closure and with $K^\times = K \setminus \{ 0 \}$ its multiplicative group.
    The univariate polynomial ring over a field $K$ is denoted as $K [x]$.
    We denote the finite field with $q$ elements by $\Fq$.

    Let $\mathcal{G}$ be a finite group, a character $\psi$ on $\mathcal{G}$ is a group homomorphism $\psi: \mathcal{G} \to U = \left\{ z \in \C \mid \abs{z} = 1 \right\}$.
    For a finite field $\F_{p^n}$, the fundamental additive character is
    \begin{equation}
        \psi_1 \left( x \right) = \exp \left( \frac{2 \cdot \pi \cdot i  \cdot \Tr_{\F_{p^n} / \Fp} (x)}{p} \right),
    \end{equation}
    where $\Tr_{\F_{p^n} / \Fp}: \F_{p^n} \to \Fp$ denotes the absolute trace function of $\F_{p^n}$.
    Also, any non-trivial additive character $\psi$ of $\F_{p^n}$ is of the form $\psi (x) = \psi_1 (b \cdot x)$ for some $b \in \F_{p^n}^\times$, see \cite[5.7.~Theorem]{Lidl-FiniteFields}.

    For a set $X$ we denote its cardinality by $\abs{X}$.
    Likewise, we denote the order of a finite group $\mathcal{G}$ as $\abs{\mathcal{G}}$.

    Let $R$ be an integral domain, i.e.\ a commutative ring with unity that does not have any zero-divisors.
    Recall that the fraction field $\Frac \left( R \right)$ of $R$ is defined via the following equivalence relation $\sim$ on $R \setminus \{ 0 \} \times R \setminus \{ 0 \}$: $(a, b) \sim (c, d) \Leftrightarrow a \cdot d = b \cdot c$ in $R$.
    Elements of fraction fields are denoted as $\frac{a}{b}$.
    Standard examples are of fraction fields are $\Q = \Frac \left( \Z \right)$, and the field of rational functions $K (x) = \Frac \big( K [x] \big)$, where $K$ is an arbitrary field.

    In the literature there exist multiple non-equivalent definitions for the degree of a rational function.
    For this work, we work with the following definition.
    \begin{defn}\label[defn]{Def: degree rational function}
        Let $K$ be a field, and let $R (x) = \frac{F (x)}{G (x)} \in K (x) \setminus \{ 0 \}$.
        The degree of $R$ is defined as
        \begin{equation*}
            \degree{R} = \degree{F} - \degree{G}.
        \end{equation*}
        For $R = 0$ we set $\degree{R} = -\infty$.
    \end{defn}

    We have the following elementary properties for the degree.
    \begin{lem}\label[lem]{Lem: degree rational function}
        Let $K$ be a field, and let $R_1 (x), R_2 (x) \in K (x)$.
        \begin{enumerate}
            \item $\degree{R_1 \cdot R_2} = \degree{R_1} + \degree{R_2}$.

            \item If $R_1 (x) = R_2 (x)$, then $\degree{R_1} = \degree{R_2}$.
        \end{enumerate}
    \end{lem}
    \begin{proof}
        Let $R_1 (x) = \frac{F_1 (x)}{G_1 (x)}$ and $R_2 (x) = \frac{F_2 (x)}{G_2 (x)}$.
        By definition and elementary properties of the polynomial degree we have
        \begin{align*}
            \degree{R_1 \cdot R_2}
            &= \degree{\frac{F_1 \cdot F_2}{G_1 \cdot G_2}} \\
            &= \degree{F_1 \cdot F_2} - \degree{G_1 \cdot G_2} \\
            &= \degree{F_1} - \degree{G_1} + \degree{F_2} - \degree{G_2} \\
            &= \degree{R_1} + \degree{R_2}.
        \end{align*}
        Note that the second equality is always satisfied if $R_1 = 0$, so for non-triviality let $R_1 \neq 0$.
        Then, $R_1 = R_2 \neq 0$ in $K (x)$ if and only if $F_1 \cdot G_2 = F_2 \cdot G_1$ in $K [x]$, but then
        \begin{align*}
            \degree{F_1 \cdot G_2} &= \degree{F_2 \cdot G_1} \\
            \Rightarrow \degree{F_1} - \degree{G_1} &= \degree{F_2} - \degree{G_2} \\
            \Rightarrow \degree{R_1} &= \degree{R_2}. \qedhere
        \end{align*}
    \end{proof}

    \subsection{Power Residue S-Boxes}\label{Sec: Power Residue S-Boxes}
    First, let us formally define power residue S-Boxes.
    \begin{defn}\label[defn]{Def: power residue S-Box}
        Let $\Fq$ be a finite field, and let $m \in \Z_{> 1}$ be such that $m \mid (q - 1)$.
        \begin{enumerate}
            \item The $m$\textsuperscript{th} power residue is defined as
            \begin{equation*}
                \legendre{x}{q}_m = x^\frac{q - 1}{m}.
            \end{equation*}

            \item Let $d \in \Z_{\geq 1}$, and let $T \in \Fq [x]$ be such that $0 \notin \imag \bigg( T \circ \legendre{x}{q}_m \bigg)$.
            The power residue S-Box of $(d, m, T)$ is defined as
            \begin{equation*}
                S (x) = x^{d} \cdot T \Bigg( \legendre{x}{q}_m \Bigg).
            \end{equation*}
        \end{enumerate}
    \end{defn}

    With an appropriate choice for $d$ and $T$ the power residue S-Box $S$ can be made a bijection on $\Fq$, see e.g.\ \cite{Shallue-Permutation,EPRINT:Szepieniec21,EC:HHLPS25}.

    For completeness, let us quickly verify that the $m$\textsuperscript{th} power residue indeed maps an element $x \in \Fqx$ to a subgroup $\mathcal{G} \subset \Fqx$ of order $m$.
    It is well-known that the multiplicative group of a finite field $\Fq$ is cyclic \cite[2.8.~Theorem]{Lidl-FiniteFields}.
    Let $g \in \Fqx$ be a cyclic generator, then for any $d \in \Z_{\geq 1}$ the element $g^d$ generates a cyclic subgroup of order $\frac{q - 1}{\gcd \left( d, q - 1 \right)}$ \cite[1.15.~Theorem]{Lidl-FiniteFields}.
    In particular, for $m \mid (q - 1)$ the subgroup $\left< g^\frac{q - 1}{m} \right> = \left\{ 1, g^\frac{q - 1}{m}, g^{2 \cdot \frac{q - 1}{m}}, \ldots, g^{(m - 1) \cdot \frac{q - 1}{m}} \right\} \subset \Fqx$ is of order $m$, and henceforth by uniqueness of cyclic subgroups $\left< g^\frac{q - 1}{m} \right> = \mathcal{G}$ \cite[1.15.~Theorem]{Lidl-FiniteFields}.
    Now let us consider the monomial map $\legendre{x}{q}_m: x \mapsto x^\frac{q - 1}{m}$.
    By division with remainder any $x \in \Fqx$ can be written as $x = g^k = g^{Q \cdot m + R}$, where $k, Q, R \in \Z$ with $0 \leq k \leq q - 2$, $0 \leq Q \leq \frac{q - 1}{m} - 1$, $0 \leq R < m$ and $k = Q \cdot m + R$.
    Therefore,
    \begin{equation}
        \legendre{x}{q}_m = x^\frac{q - 1}{m} = g^{\left( Q \cdot m + R \right) \cdot \frac{q - 1}{m}} = g^{Q \cdot (q - 1)} \cdot g^{R \cdot \frac{q - 1}{m}} = g^{R \cdot \frac{q - 1}{m}} \in \left< g^\frac{q - 1}{m} \right> = \mathcal{G}.
    \end{equation}
    I.e., we can view the $m$\textsuperscript{th} power residue symbol as the projection of $\Fqx$ onto $\mathcal{G}$.
    Moreover, by construction we have for any $y \in \mathcal{G}$ that
    \begin{equation}
        \abs{\legendre{x}{q}_m^{-1} (y)} = \frac{q - 1}{m}.
    \end{equation}

    We provide some examples of power residue S-Boxes which have been introduced in the literature.
    Conditions for the S-Boxes to induce permutations can be found in the respective references.
    \begin{ex}
        Let $\Fq$ be a finite field, and let $d, m \in \Z_{\geq 1}$ be such that $m \mid (q - 1)$.
        \begin{enumerate}
            \item \cite[Theorem~1.21]{Shallue-Permutation} $S (x) = x \cdot \bigg( \legendre{x}{q}_m - a \bigg)$, where $a \in \Fq \setminus \imag \bigg( \legendre{x}{q}_m \bigg)$.

            \item \cite[\S 3]{EPRINT:Szepieniec21} $S (x) = x^d \cdot \legendre{x}{q}_2$, where $q$ is odd.

            \item \cite[Proposition~3]{ToSC:GKRS22} $S (x) = x^d \cdot \bigg( \legendre{x}{q}_2 + a \bigg)$, where $q$ is odd and $a \in \Fq \setminus \{ \pm 1 \}$.

            \item \cite[\S 3]{EC:HHLPS25} $S (x) = x^{q - 2} \cdot T \bigg( \legendre{x}{q}_m \bigg)$, $0 \notin \imag \bigg( T \circ \legendre{x}{q}_m \bigg)$.
        \end{enumerate}
    \end{ex}

    \subsection{Walsh Transform}
    The Walsh transform is an important mathematical tool to study cryptographic properties of S-Boxes.
    In particular, it encompasses the resistance of an S-Box against linear cryptanalysis.
    \begin{defn}\label[defn]{Def: Walsh transform}
        Let $\Fq$ be a finite field, let $n, m \in \Z_{\geq 1}$, let $\psi: \Fq \to \C$ be a non-trivial additive character, let $F: \Fqn \to \Fqm$ be a function, and let $\mathbf{a} \in \Fqn$ and $\mathbf{b} \in \Fqm$.
        \begin{enumerate}
            \item The Walsh transform for the character $\psi$ of the linear approximation $(\mathbf{a}, \mathbf{b})$ of $F$ is defined as
            \begin{equation*}
                \mathcal{W}_F (\psi, \mathbf{a}, \mathbf{b}) = \sum_{\mathbf{x} \in \Fqn} \psi \big( \mathbf{a}^\intercal \mathbf{x} + \mathbf{b}^\intercal F (\mathbf{x}) \big).
            \end{equation*}

            \item Let $X \subset \Fqn$.
            The Walsh transform over $X$ for the character $\psi$ of the linear approximation $(\mathbf{a}, \mathbf{b})$ of $F$ is defined as
            \begin{equation*}
                \mathcal{W}_F (\psi, X, \mathbf{a}, \mathbf{b}) = \sum_{\mathbf{x} \in X} \psi \big( \mathbf{a}^\intercal \mathbf{x} + \mathbf{b}^\intercal F (\mathbf{x}) \big).
            \end{equation*}
        \end{enumerate}
    \end{defn}

    In the literature, the values $\mathcal{W}_F (\psi, \mathbf{a}, \mathbf{b})$ are also known as the Walsh spectrum of $F$.

    \begin{rem}
        \begin{enumerate}
            \item In more generality, the Walsh transform can also be defined via two non-trivial additive characters $\psi, \phi: \Fq \to \C$
            \begin{equation*}
                \mathcal{W}_F (\psi, \phi, \mathbf{a}, \mathbf{b}) = \sum_{\mathbf{x} \in \Fqn} \psi \big( \mathbf{a}^\intercal \mathbf{x} \big) \cdot \phi \big( \mathbf{b}^\intercal F (\mathbf{x}) \big).
            \end{equation*}
            Since non-trivial additive characters are equivalent up to a multiplicative constant, i.e.\ there exists $c \in \Fqx$ such that $\phi (x) = \psi (c \cdot x)$ \cite[5.7.~Theorem]{Lidl-FiniteFields}, the more general definition is equivalent to our one up to a multiplicative constant.

            \item In modern formulations of linear cryptanalysis the correlation of a function $F$ at the linear approximation $(\mathbf{a}, \mathbf{b})$ is defined as, see e.g.\ \cite[Definition~3.3]{AC:Beyne21} and the equation thereafter,
            \begin{equation*}
                \CORR_F (\psi, \phi, \mathbf{a}, \mathbf{b}) = \frac{\mathcal{W}_F \big( \psi, \overline{\phi}, \mathbf{a}, \mathbf{b} \big)}{q^n},
            \end{equation*}
            where $\overline{z}$ denotes complex conjugation. I.e., $\overline{\phi (x)} = \phi (-x)$.
            Hence, up to a multiplicative constant the correlation also reduces to our definition of the Walsh transform.
        \end{enumerate}
    \end{rem}

    \subsection{Exponential Sums With Rational Arguments}
    Let $X \subset \Fq$, let $F: \Fq \to \Fq$ be a function, and let $\psi: \Fq \to \C$ be an additive character.
    An important problem in number theory is the estimation of the character sum $\abs{\sum_{x \in X} \psi \big( F (x) \big)}$.
    A fundamental result for sums with polynomial arguments is Weil's bound \cite{Weil-Bound}.
    \begin{thm}[{\cite[5.38.~Theorem]{Lidl-FiniteFields}}]\label[thm]{Th: Weil bound}
        Let $\Fq$ be a finite field, let $\psi: \Fq \to \C$ be a non-trivial additive character, and let $f \in \Fq [x]$ be such that $\gcd \big( \degree{f}, q \big) = 1$.
        Then
        \begin{equation*}
            \abs{\sum_{x \in \Fq} \psi \big( f (x) \big)} \leq \big( \degree{f} - 1 \big) \cdot q^\frac{1}{2}.
        \end{equation*}
    \end{thm}

    Moreno \& Moreno generalized Weil's bound to rational functions $R \in \Fq (x)$.
    In case that $R$ has a pole at $a \in \Fq$, then $a$ is ignored in the summation.
    \begin{thm}[{\cite[Theorem~2]{Moreno-Weil-I}}]\label[thm]{Th: Weil bound for rational functions}
        Let $\Fq$ be a finite field of characteristic $p$, and let $R (x) = \frac{F (x)}{G (x)} \in \Fq (x)$ be a rational function such that
        \begin{equation*}
            R (x) \neq h (x)^p - h (x)
        \end{equation*}
        for any $h (x) \in \overline{\Fq} (x)$.
        Let $\mathcal{P} \subset \overline{\Fq}$ be the set of distinct roots of $G (x)$ over the algebraic closure, let $s = \abs{\mathcal{P}}$, and let $\psi: \Fq \to \C$ be a non-trivial additive character.
        Then
        \begin{equation*}
            \abs{\sum_{x \in \Fq \setminus \mathcal{P}} \psi \big( R (x) \big)}
            \leq \left( \max \left\{ \degree{F}, \degree{G} \right\} + s^\ast - 2 \right) \cdot q^\frac{1}{2} + \delta,
        \end{equation*}
        where $s^\ast =
        \begin{rcases}
            \begin{cases}
                s, & \degree{F} \leq \degree{G}, \\
                s + 1, & \degree{F} > \degree{G}
            \end{cases}
        \end{rcases}
        $ and $\delta =
        \begin{rcases}
            \begin{cases}
                1, & \degree{F} \leq \degree{G}, \\
                0, & \degree{F} > \degree{G}
            \end{cases}
        \end{rcases}
        .$
    \end{thm}

    One of the simplest non-trivial examples for character sums with rational arguments are so-called \emph{Kloosterman sums} \cite{Kloosterman-Sum}.
    \begin{defn}
        Let $\Fq$ be a finite field, let $\psi: \Fq \to \C$ be a non-trivial additive character, and let $a, b \in \Fq$.
        \begin{enumerate}
            \item Then the sum
            \begin{equation*}
                K (\psi, a, b) = \sum_{x \in \Fqx} \psi \left( a \cdot x + b \cdot x^{-1} \right)
            \end{equation*}
            is called Kloosterman sum.

            \item Let $X \subset \Fqx$ be non-empty, then the sum
            \begin{equation*}
                K (\psi, X, a, b) = \sum_{x \in X} \psi \left( a \cdot x + b \cdot x^{-1} \right)
            \end{equation*}
            is called Kloosterman sum over $X$.
        \end{enumerate}
    \end{defn}

    It is easy to see that \Cref{Th: Weil bound for rational functions} implies that $\abs{K (\psi, a, b)} \leq 2 \cdot q^\frac{1}{2}$, when $a \cdot b \neq 0$.
    However, we note that this estimation can also be proven with elementary techniques, see for example \cite[5.45.~Theorem]{Lidl-FiniteFields}.

    \section{Walsh Spectrum of Power Residue S-Boxes}
    We now estimate the Walsh spectrum of power residue S-Boxes $(d, m, T)$ with
    \begin{enumerate}[label=(\roman*)]
        \item $d = q - 2$ (\Cref{Sec: inverse power residue S-Box}),

        \item\label{Item: power permutation} $d > 1$, $d \cdot m < q$ and $\gcd \left( d, q \right) = 1$ (\Cref{Sec: power permutation resiude S-Box}), and

        \item $d = 1$ (\Cref{Sec: linear resiude S-Box}).
    \end{enumerate}
    The estimations are applications of the Weil bounds (\Cref{Th: Weil bound,Th: Weil bound for rational functions}).

    With the estimations for \ref{Item: power permutation} we can also derive the Walsh spectrum of a Legendre symbol-based S-Box (\Cref{Sec: Legendre symbol-based S-Box}) introduced by Grassi et al.\ \cite[\S 4.1]{ToSC:GKRS22}.

    \subsection{\texorpdfstring{$\mathbf{d = q - 2}$}{d = q - 2}}\label{Sec: inverse power residue S-Box}
    For a subgroup $\mathcal{G} \subset \Fqx$ the Kloosterman sum $K (\psi, \mathcal{G}, a, b)$ can be rewritten as an exponential sum over $\Fqx$ with rational arguments.
    \begin{lem}\label[lem]{Lem: Kloosterman sum over subgroup}
        Let $\Fq$ be a finite field, let $\mathcal{G} \subset \Fqx$ be a subgroup, let $\psi: \Fq \to \C$ be a non-trivial additive character, and let $a, b \in \Fq$.
        Then
        \begin{equation*}
            K (\psi, \mathcal{G}, a, b) = \frac{\abs{\mathcal{G}}}{q - 1} \cdot \sum_{x \in \Fqx} \psi \left( a \cdot x^\frac{q - 1}{\abs{\mathcal{G}}} + b \cdot x^{-\frac{q - 1}{\abs{\mathcal{G}}}} \right).
        \end{equation*}
    \end{lem}
    \begin{proof}
        Recall from \Cref{Sec: Power Residue S-Boxes}, for the $\abs{\mathcal{G}}$\textsuperscript{th} power residue symbol $\legendre{x}{q}_{\abs{\mathcal{G}}} : x \mapsto x^\frac{q - 1}{\abs{\mathcal{G}}}$ we have
        \begin{enumerate}[label=(\roman*)]
            \item for all $x \in \Fqx$ that $\legendre{x}{q}_{\abs{\mathcal{G}}} \in \mathcal{G}$, and

            \item for all $y \in \mathcal{G}$ that $\abs{\legendre{x}{q}_{\abs{\mathcal{G}}}^{-1} (y)} = \frac{q - 1}{\abs{\mathcal{G}}}$.
        \end{enumerate}
        Therefore
        \begin{align*}
            \sum_{x \in \Fqx} \psi \Bigg( a \cdot \legendre{x}{q}_{\abs{\mathcal{G}}} + b \cdot \legendre{x}{q}_{\abs{\mathcal{G}}}^{-1} \Bigg)
            &= \sum_{x \in \Fqx} \psi \left( a \cdot x^\frac{q - 1}{\abs{\mathcal{G}}} + b \cdot x^{-\frac{q - 1}{\abs{\mathcal{G}}}} \right) \\
            &= \frac{q - 1}{\abs{\mathcal{G}}} \cdot \sum_{x \in \mathcal{G}} \psi \left( a \cdot x + b \cdot x^{-1} \right). \qedhere
        \end{align*}
    \end{proof}

    Next we answer \cite[Conjecture~1]{EC:HHLPS25} in the affirmative.
    \begin{thm}\label[thm]{Th: Kloosterman spectrum over subgroup}
        Let $\Fq$ be a finite field, let $\mathcal{G} \subset \Fqx$ be a subgroup, let $\psi: \Fq \to \C$ be a non-trivial additive character, and let $a, b \in \Fq$.
        Then
        \begin{equation*}
            \abs{K (\psi, \mathcal{G}, a, b)} \leq
            \begin{dcases}
                \abs{\mathcal{G}}, & a, b = 0, \\
                \frac{\abs{\mathcal{G}}}{q - 1} + \left( 1 - \frac{\abs{\mathcal{G}}}{q - 1} \right) \cdot q^\frac{1}{2}, &
                \left\{
                \begin{array}{c}
                    a \neq 0, b = 0, \\
                    a = 0, b \neq 0
                \end{array}
                \right\}, \\
                2 \cdot q^\frac{1}{2}, & a \cdot b \neq 0.
            \end{dcases}
        \end{equation*}
    \end{thm}
    \begin{proof}
        The case $a, b = 0$ is trivial.
        For $a \neq 0$, $b = 0$ or $a = 0$, $b \neq 0$, first note that due to $\mathcal{G}$ being a group, we have
        \begin{equation*}
            \abs{K (\psi, \mathcal{G}, 0, b)}
            = \sum_{x \in \mathcal{G}} \psi \left( b \cdot x^{-1} \right)
            = \sum_{x \in \mathcal{G}} \psi \left( b \cdot x \right)
            = \abs{K (\psi, \mathcal{G}, b, 0)}.
        \end{equation*}
        Then, by application of \Cref{Lem: Kloosterman sum over subgroup} and the Weil bound (\Cref{Th: Weil bound})
        \begin{align*}
            \abs{K (\psi, \mathcal{G}, b, 0)} = \abs{K (\psi, \mathcal{G}, 0, b)}
            &= \frac{\abs{\mathcal{G}}}{q - 1} \cdot \abs{\sum_{x \in \Fqx} \psi \left( b \cdot x^\frac{q - 1}{\abs{\mathcal{G}}} \right)} \\
            &= \frac{\abs{\mathcal{G}}}{q - 1} \cdot \abs{-1 + \sum_{x \in \Fq} \psi \left( b \cdot x^\frac{q - 1}{\abs{\mathcal{G}}} \right)} \\
            & \leq \frac{\abs{\mathcal{G}}}{q - 1} \cdot \left( 1 + \abs{\sum_{x \in \Fq} \psi \left( b \cdot x^\frac{q - 1}{\abs{\mathcal{G}}} \right)} \right) \\
            &\leq \frac{\abs{\mathcal{G}}}{q - 1} \cdot \left( 1 +  \left( \frac{q - 1}{\abs{\mathcal{G}}} - 1 \right) \cdot q^\frac{1}{2} \right).
        \end{align*}

        For $a \cdot b \neq 0$, we will apply \Cref{Lem: Kloosterman sum over subgroup} followed by \Cref{Th: Weil bound for rational functions}.
        To verify the conditions of \Cref{Th: Weil bound for rational functions}, let us assume that
        \begin{align}
            a \cdot x^\frac{q - 1}{\abs{\mathcal{G}}} + b \cdot x^{-\frac{q - 1}{\abs{\mathcal{G}}}}
            &= \left( \frac{F (x)}{G (x)} \right)^p - \frac{F (x)}{G (x)} \\
            \Longleftrightarrow \frac{a \cdot x^{2 \cdot \frac{q - 1}{\abs{\mathcal{G}}}} + b}{x^\frac{q - 1}{\abs{\mathcal{G}}}}
            &= \frac{F (x) \cdot \left( F (x)^{p - 1} - G (x)^{p - 1} \right)}{G (x)^p}, \label{Equ: rational function}
        \end{align}
        where $p$ denotes the characteristic of $\Fq$ and $F, G \in \overline{\Fq} [x]$.
        The rational function on the left-hand side of \Cref{Equ: rational function} is non-zero and has degree $\frac{q - 1}{\abs{\mathcal{G}}}$, so by \Cref{Lem: degree rational function} we must have that $\degree{F} > \degree{G}$.\footnote{
            If $\degree{F} \leq \degree{G}$, then the rational function on the right-hand side of \Cref{Equ: rational function} has degree $\leq 0$.
            }
        Therefore, we yield the degree equality of rational functions
        \begin{equation*}
            \frac{q - 1}{\abs{\mathcal{G}}}
            = \degree{F} + (p - 1) \cdot \degree{F} - p \cdot \degree{G}
            = p \cdot \big( \degree{F} - \degree{G} \big).
        \end{equation*}
        Hence, we have that $p \mid \frac{q - 1}{\abs{\mathcal{G}}}$ which contradicts $\gcd \left( p, q - 1 \right) = 1$.
        So, the asserted rational function from \Cref{Equ: rational function} cannot exist, and we can apply \Cref{Th: Weil bound for rational functions}.
        Then
        \begin{align*}
            \abs{K (\psi, \mathcal{G}, a, b)}
            &= \frac{\abs{\mathcal{G}}}{q - 1} \cdot \abs{\sum_{x \in \Fqx} \psi \left( a \cdot x^\frac{q - 1}{\abs{\mathcal{G}}} + b \cdot x^{-\frac{q - 1}{\abs{\mathcal{G}}}} \right)} \\
            &= \frac{\abs{\mathcal{G}}}{q - 1} \cdot \abs{\sum_{x \in \Fqx} \psi \left( \frac{a \cdot x^{2 \cdot \frac{q - 1}{\abs{\mathcal{G}}}} + b}{x^\frac{q - 1}{\abs{\mathcal{G}}}} \right)} \\
            &\leq \frac{\abs{\mathcal{G}}}{q - 1} \cdot \left( 2 \cdot \frac{q - 1}{\abs{\mathcal{G}}} + 2 - 2 \right) \cdot q^\frac{1}{2} + 0
            = 2 \cdot q^\frac{1}{2}. \qedhere
        \end{align*}
    \end{proof}

    As corollary, we yield the Walsh spectrum of power residue S-Boxes.
    \begin{cor}\label[cor]{Cor: Walsh spectrum power residue S-Box inverse}
        Let $\Fq$ be a finite field, let $m \in \Z_{> 1}$ be such that $m \mid (q - 1)$, and let $S: \Fq \to \Fq$ be the power residue S-Box of $(q - 2, m, T)$.
        Let $\psi: \Fq \to \C$ be a non-trivial additive character, and let $a, b \in \Fq$.
        Then
        \begin{equation*}
            \abs{\mathcal{W}_S (\psi, a, b)} \leq
            \begin{dcases}
                q, & a = b = 0, \\
                0, & a \neq 0, b = 0, \\
                \left( m - 1 \right) \cdot q^\frac{1}{2} + 2, & a = 0, b \neq 0, \\
                2 \cdot m \cdot q^\frac{1}{2} + 1, & a \cdot b \neq 0.
            \end{dcases}
        \end{equation*}
    \end{cor}
    \begin{proof}
        We recapitulate the reduction of $\abs{\mathcal{W}_S (\psi, a, b)}$ to a Kloosterman sum from \cite[\S 5.1]{EC:HHLPS25}.
        Let $g \in \Fqx$ be a generator, and for $0 \leq r \leq m - 1$ let $N_r = \big\{ g^r, g^{m + r}, g^{2 \cdot m + r}, \ldots, \allowbreak g^{q - 1 - m + r} \big\}$.
        Then
        \begin{align*}
            \mathcal{W}_S (\psi, a, b)
            &= \sum_{x \in \Fq} \psi \big( a \cdot x + b \cdot S (x) \big) \\
            &= 1 + \sum_{r = 0}^{m - 1} \sum_{x \in N_r} \psi \big( a \cdot x + b \cdot S (x) \big) \\
            &= 1 + \sum_{r = 0}^{m - 1} \sum_{x \in N_r} \psi \Bigg( a \cdot x + b \cdot x^{-1} \cdot T \Bigg( \legendre{x}{q}_m \Bigg) \Bigg) \\
            &= 1 + \sum_{r = 0}^{m - 1} \sum_{x \in N_r} \psi \Bigg( a \cdot x + b \cdot x^{-1} \cdot T \bigg( g^{\left( k \cdot m + r \right) \cdot \frac{q - 1}{m}} \bigg) \Bigg) \\
            &= 1 + \sum_{r = 0}^{m - 1} \sum_{x \in N_r} \psi \Bigg( a \cdot x + b \cdot x^{-1} \cdot T \bigg( g^{r \cdot \frac{q - 1}{m}} \bigg) \Bigg) \\
            &= 1 + \sum_{r = 0}^{m - 1} \sum_{x \in N_0} \psi \Bigg( a \cdot x \cdot g^r + b \cdot x^{-1} \cdot g^{-r} \cdot T \bigg( g^{r \cdot \frac{q - 1}{m}} \bigg) \Bigg).
        \end{align*}
        (We note that Ha et al.\ forgot the term $1$, which corresponds to $\psi (0)$, in their computation.)
        For $r$ fixed, we can consider the terms $g^r$ and $g^{-r} \cdot T \left( g^{r \cdot \frac{q - 1}{m}} \right)$ as non-zero constants.\footnote{Recall from \Cref{Def: power residue S-Box} that $0 \notin \imag \bigg( T \circ \legendre{x}{q}_m \bigg)$.}
        Therefore, when passing to the absolute value we yield with the triangular inequality that
        \begin{equation*}
            \abs{\mathcal{W}_S (\psi, a, b)}
            \leq 1 + \sum_{r = 0}^{m - 1} \abs{\sum_{x \in N_0} \psi \left( a_r \cdot x + b_r \cdot x^{-1} \right)}
            = 1 + \sum_{r = 0}^{m - 1} \abs{K (\psi, N_0, a_r, b_r)},
        \end{equation*}
        where $a_r = a \cdot g^r$ and $b_r = b \cdot g^{-r} \cdot T \left( g^{r \cdot \frac{q - 1}{m}} \right)$.
        On the other hand, the set $N_0$ is also a group, namely $N_0 = \left< g^m \right> \subset \Fqx$.
        So, we can apply \Cref{Th: Kloosterman spectrum over subgroup} to estimate $\abs{K (\psi, N_0, a_r, b_r)}$.

        Now we do a case distinction on $a, b \in \Fq$.
        \begin{itemize}
            \item The case $a = b = 0$ is trivial.

            \item The case $a \neq 0$, $b = 0$ is well-known \cite[5.31.~Theorem]{Lidl-FiniteFields}.

            \item The case $a = 0$, $b \neq 0$ follows with \Cref{Th: Kloosterman spectrum over subgroup}
            \begin{equation*}
                \abs{K (\psi, N_0, 0, b_r)}
                \leq \frac{1}{m} + \left( 1 - \frac{1}{m} \right) \cdot q^\frac{1}{2}.
            \end{equation*}

            \item The case $a \cdot b \neq 0$ also follows with \Cref{Th: Kloosterman spectrum over subgroup}
            \begin{equation*}
                \abs{K (\psi, N_0, a_r, b_r)}
                \leq 2 \cdot q^\frac{1}{2}. \qedhere
            \end{equation*}
        \end{itemize}
    \end{proof}
    \begin{rem}
        If $S$ is a permutation, then the case $a = 0$, $b \neq 0$ improves to $\mathcal{W}_S (\psi, 0, b) = 0$.
    \end{rem}

    \begin{ex}[\Polocolo]\label[ex]{Ex: Polocolo}
        Let $\Fq$ be a prime field such that $2^n \mid (q - 1)$, and let $S: \Fq \to \Fq$ be a $\big( q - 2, 2^n, T \big)$ power residue S-Box which also induces a permutation on $\Fq$.
        I.e., $S$ is a \Polocolo{} S-Box \cite[\S 3]{EC:HHLPS25}.
        Then the Walsh spectrum of $S$ is bounded by
        \begin{equation*}
            \abs{\mathcal{W}_S (\psi, a, b)} \leq
            \begin{dcases}
                q, & a = b = 0, \\
                0, &
                \left\{
                \begin{array}{c}
                    a \neq 0, b = 0, \\
                    a = 0, b \neq 0
                \end{array}
                \right\}, \\
                2^{n + 1} \cdot q^\frac{1}{2} + 1, & a \cdot b \neq 0.
            \end{dcases}
        \end{equation*}
    \end{ex}

    The \Polocolo{} designers \cite[\S 5.1]{EC:HHLPS25} conjectured for any non-trivial additive characters $\psi, \phi: \Fq \to \C$ that
    \begin{equation}
        \max_{a, b \in \Fqx} \abs{\CORR_S (\psi, \phi, a, b)} \leq \frac{2^{n + 2}}{q^\frac{1}{2}}.
    \end{equation}
    We note that this inequality follows trivially from \Cref{Ex: Polocolo}.

    \subsubsection{\texorpdfstring{$\mathbf{d = e \cdot (q - 2)}$}{d = e * (q - 2)}}
    It is possible to generalize the previous arguments to $(e \cdot (q - 2), m, T)$ power residue S-Boxes.
    Since the arguments are almost identical to the special case $e = 1$, we only sketch the full proof.
    \begin{cor}
        Let $\Fq$ be a finite field, let $e, m \in \Z_{> 1}$ be such that $m \mid (q - 1)$, and let $S: \Fq \to \Fq$ be the power residue S-Box of $(e \cdot (q - 2), m, T)$.
        Let $\psi: \Fq \to \C$ be a non-trivial additive character, let $a, b \in \Fq$, and let $\mathcal{G} \subset \Fqx$ be a subgroup.
        \begin{enumerate}
            \item $\sum_{x \in \mathcal{G}} \psi \left( a \cdot x + b \cdot x^{-e} \right)
            = \frac{\abs{\mathcal{G}}}{q - 1} \cdot \sum_{x \in \Fqx} \psi \left( a \cdot x^\frac{q - 1}{\abs{\mathcal{G}}} + b \cdot x^{-e \cdot \frac{q - 1}{\abs{\mathcal{G}}}} \right)$.

            \item $\abs{\sum_{x \in \mathcal{G}} \psi \left( a \cdot x + b \cdot x^{-e} \right)} \leq
            \begin{dcases}
                \abs{\mathcal{G}}, & a, b = 0, \\
                \frac{\abs{\mathcal{G}}}{q - 1} + \left( 1 - \frac{\abs{\mathcal{G}}}{q - 1} \right) \cdot q^\frac{1}{2}, & a \neq 0, b = 0, \\
                \frac{\abs{\mathcal{G}}}{q - 1} + \left( e - \frac{\abs{\mathcal{G}}}{q - 1} \right) \cdot q^\frac{1}{2}, & a = 0, b \neq 0, \\
                (e + 1) \cdot q^\frac{1}{2}, & a \cdot b \neq 0.
            \end{dcases}
            $

            \item $\mathcal{W}_S (\psi, a, b) \leq
            \begin{dcases}
                q, & a, b = 0, \\
                0, & a \neq 0, b = 0, \\
                (e \cdot m - 1) \cdot q^\frac{1}{2} + 2, & a = 0, b \neq 0, \\
                (e + 1) \cdot m \cdot q^\frac{1}{2}, & a \cdot b \neq 0.
            \end{dcases}
            $
        \end{enumerate}
    \end{cor}
    \begin{proof}[Proof (Sketch)]
        Claim (1) follows identical to \Cref{Lem: Kloosterman sum over subgroup}.

        For claim (2), we have to settle a new case namely $a = 0$, $b \neq 0$.
        Since $x \mapsto x^{-1}$ is a group homomorphism on $\mathcal{G}$, we have that
        \begin{equation*}
            \abs{\sum_{x \in \mathcal{G}} \psi \left( b \cdot x^{-e} \right)}
            = \abs{\sum_{y \in \mathcal{G}} \psi \left( b \cdot y^{e} \right)}
            = \frac{\abs{\mathcal{G}}}{q - 1} \cdot \abs{-1 + \sum_{y \in \Fq} \psi \left( b \cdot y^{e \cdot \frac{q - 1}{\mathcal{G}}} \right)}.
        \end{equation*}
        Now we apply the triangular inequality followed by the Weil bound (\Cref{Th: Weil bound}).
        For $a \cdot b \neq 0$, analogously to the proof of \Cref{Th: Kloosterman spectrum over subgroup}, we assume that there exists rational functions $F, G \in \overline{\Fq} (x)$ such that
        \begin{align*}
            a \cdot x^\frac{q - 1}{\abs{\mathcal{G}}} + b \cdot x^{-e \cdot \frac{q - 1}{\abs{\mathcal{G}}}}
            &= \left( \frac{F (x)}{G (x)} \right)^p - \frac{F (x)}{G (x)} \\
            \Longleftrightarrow \frac{a \cdot x^{(e + 1) \cdot \frac{q - 1}{\abs{\mathcal{G}}}} + b}{x^{e \cdot \frac{q - 1}{\abs{\mathcal{G}}}}}
            &= \frac{F (x) \cdot \left( F (x)^{p - 1} - G (x)^{p - 1} \right)}{G (x)^p},
        \end{align*}
        where $p$ is the characteristic of $\Fq$.
        Then, we can again establish the degree equality of rational functions $\frac{q - 1}{\abs{\mathcal{G}}} = p \cdot \big( \degree{F} - \degree{G} \big)$, which contradicts $\gcd \left( p, q - 1 \right) = 1$.
        So, we can apply \Cref{Th: Weil bound for rational functions} to obtain the estimation.

        Claim (3) follows identical to \Cref{Cor: Walsh spectrum power residue S-Box inverse}.
    \end{proof}
    \begin{rem}
        If $x \mapsto x^e$ induces a permutation monomial on $\Fq$, i.e.\ if $\gcd \left( e, q - 1 \right) = 1$, then $x \mapsto x^{e \cdot (q - 2)}$ is also a permutation monomial \cite[7.8.~Theorem]{Lidl-FiniteFields}.
        Now let $g \in \Fqx$ be a generator, and let $m \in \Z_{> 1}$ be such that $m \mid (q - 1)$.
        Then, $\left< g^m \right> \subset \Fqx$ is a subgroup of order $\frac{q - 1}{m}$.
        Obviously, $g^{e \cdot m} \in \left< g^m \right>$.
        In addition, $\left< g^{e \cdot m} \right> \subset \Fqx$ is a subgroup of order $\frac{q - 1}{\gcd \left( e \cdot m, q - 1 \right)} = \frac{q - 1}{m}$ \cite[1.15.~Theorem]{Lidl-FiniteFields}.
        I.e., a permutation monomial on $\Fq$ also induces a permutation on any subgroup $\mathcal{G} \subset \Fqx$.
        In this case the estimations for $a = 0$, $b \neq 0$ improve to the ones of $a \neq 0$, $b = 0$.
    \end{rem}

    \subsection{\texorpdfstring{$\mathbf{d > 1}$}{d > 1} and \texorpdfstring{$\mathbf{d \cdot m < q}$}{d * m < q}}\label{Sec: power permutation resiude S-Box}
    Next we investigate power residue S-Boxes with $d \neq q - 2$.
    First we state an analog of \Cref{Lem: Kloosterman sum over subgroup} for the Walsh spectrum.
    \begin{lem}\label[lem]{Lem: character sum over subgroup}
        Let $\Fq$ be a finite field, let $\mathcal{G} \subset \Fqx$ be a subgroup, let $\psi: \Fq \to \C$ be a non-trivial additive character, let $f \in \Fq [x]$, and let $a, b \in \Fq$.
        Then
        \begin{equation*}
            \mathcal{W}_f (\psi, \mathcal{G}, a, b) = \frac{\abs{\mathcal{G}}}{q - 1} \cdot \sum_{x \in \Fqx} \psi \Bigg( a \cdot \legendre{x}{q}_{\abs{\mathcal{G}}} + b \cdot f \Bigg( \legendre{x}{q}_{\abs{\mathcal{G}}} \Bigg) \Bigg).
        \end{equation*}
    \end{lem}
    \begin{proof}
        Again, we utilize that $\abs{\legendre{x}{q}_{\abs{\mathcal{G}}}^{-1} (y)} = \frac{q - 1}{\abs{\mathcal{G}}}$ for any $y \in \mathcal{G}$.
        Then
        \begin{equation*}
            \sum_{x \in \Fqx} \psi \Bigg( a \cdot \legendre{x}{q}_{\abs{\mathcal{G}}} + b \cdot f \Bigg( \legendre{x}{q}_{\abs{\mathcal{G}}} \Bigg) \Bigg)
            = \frac{q - 1}{\abs{\mathcal{G}}} \cdot \sum_{x \in \mathcal{G}} \psi \big( a \cdot x + b \cdot f (x) \big). \qedhere
        \end{equation*}
    \end{proof}

    Next we require an analog \Cref{Th: Kloosterman spectrum over subgroup}.
    However, for an arbitrary $f \in \Fq [x]$ the situation is more subtle.
    Let us provide a quick illustration.
    Via the previous lemma we have that
    \begin{equation}\label{Equ: Walsh transform over subgroup}
        \mathcal{W}_f (\psi, \mathcal{G}, a, b) = \frac{\abs{\mathcal{G}}}{q - 1} \cdot \sum_{x \in \Fqx} \psi \bigg( a \cdot x^\frac{q - 1}{\abs{\mathcal{G}}} + b \cdot f \Big( x^\frac{q - 1}{\abs{\mathcal{G}}} \Big) \bigg).
    \end{equation}
    Recall that by Lagrange interpolation any function $F: \Fq \to \Fq$ can be represented by a unique polynomial $F \in \Fq [x] / \left( x^q - x \right)$.
    Therefore, to apply the Weil bound (\Cref{Th: Weil bound}) to \Cref{Equ: Walsh transform over subgroup} we have to ensure that
    \begin{equation}
        a \cdot x^\frac{q - 1}{\abs{\mathcal{G}}} + b \cdot f \Big( x^\frac{q - 1}{\abs{\mathcal{G}}} \Big) \not\equiv c \mod \left( x^q - x \right),
    \end{equation}
    where $c \in \Fq$.
    Unfortunately, this exactly happens for power residue S-Boxes with $d = 1$.
    \begin{ex}\label[ex]{Ex: d = 1}
        Let $\Fq$ be a finite field with $q$ odd, let $S (x) = x \cdot \legendre{x}{q}_2$, and let $N_0 \subset \Fqx$ be the group of squares.
        I.e., $N_0 = \left\{ 1, g^2, g^4, \dots, g^{q - 1 - 2} \right\}$, where $g \in \Fqx$ is a generator.
        Then
        \begin{align*}
            \mathcal{W}_S (\psi, N_0, a, -a)
            &= \sum_{x \in N_0} \psi \Bigg(a \cdot x - a \cdot x \cdot \legendre{x}{q}_2 \Bigg) \\
            &= \sum_{x \in N_0} \psi \left(a \cdot x - a \cdot x \cdot x^\frac{q - 1}{2} \right) \\
            &= \sum_{x \in N_0} \psi \left( a \cdot x - a \cdot x \right)
            = \frac{q - 1}{2}.
        \end{align*}
    \end{ex}

    Therefore, to ensure that the polynomial from the right-hand side of \Cref{Equ: Walsh transform over subgroup} is non-constant, we must guarantee that the map $x \mapsto a \cdot x + b \cdot f (x)$ is non-constant on $\mathcal{G}$ for all $a \in \Fq$ and $b \in \Fqx$.
    This technical condition will be reflected in the following theorem.
    \begin{thm}\label[thm]{Th: Walsh spectrum over subgroup}
        Let $\Fq$ be a finite field, let $\mathcal{G} \subset \Fqx$ be a subgroup, let $\psi: \Fq \to \C$ be a non-trivial additive character, let $f \in \Fq [x]$ be such that $\degree{f} > 1$ and $\gcd \left( \degree{f}, q \right) = 1$, and let $a, b \in \Fq$.
        Assume for all $a \in \Fq$ and $b \in \Fqx$ that $x \mapsto a \cdot x + b \cdot f (x)$ is non-constant on $\mathcal{G}$.
        Then
        \begin{equation*}
            \abs{\mathcal{W}_f (\psi, \mathcal{G}, a, b)} \leq
            \begin{dcases}
                \abs{\mathcal{G}}, & a, b = 0, \\
                \frac{\abs{\mathcal{G}}}{q - 1} + \left( 1 - \frac{\abs{\mathcal{G}}}{q - 1} \right) \cdot q^\frac{1}{2}, & a \neq 0, b = 0, \\
                \frac{\abs{\mathcal{G}}}{q - 1} + \left( \degree{f} - \frac{\abs{\mathcal{G}}}{q - 1} \right) \cdot q^\frac{1}{2}, & a = 0, b \neq 0, \\
                \frac{\abs{\mathcal{G}}}{q - 1} + \left( \degree{f} - \frac{\abs{\mathcal{G}}}{q - 1} \right) \cdot q^\frac{1}{2}, & a \cdot b \neq 0.
            \end{dcases}
        \end{equation*}
    \end{thm}
    \begin{proof}
        The case $a, b = 0$ is trivial.
        For $a \neq 0$, $b = 0$ we have by application of \Cref{Lem: character sum over subgroup} and the Weil bound (\Cref{Th: Weil bound})
        \begin{align*}
            \abs{\mathcal{W}_f (\psi, \mathcal{G}, a, 0)}
            &= \frac{\abs{\mathcal{G}}}{q - 1} \cdot \abs{\sum_{x \in \Fqx} \psi \Big( a \cdot x^\frac{q - 1}{\abs{\mathcal{G}}} \Big)} \\
            &= \frac{\abs{\mathcal{G}}}{q - 1} \cdot \abs{-\psi (0) + \sum_{x \in \Fq} \psi \Big( a \cdot x^\frac{q - 1}{\abs{\mathcal{G}}} \Big)} \\
            &\leq \frac{\abs{\mathcal{G}}}{q - 1} \cdot \left( 1 + \abs{\sum_{x \in \Fq} \psi \Big( a \cdot x^\frac{q - 1}{\abs{\mathcal{G}}} \Big)} \right) \\
            &\leq \frac{\abs{\mathcal{G}}}{q - 1} \cdot \left( 1 + \left( \frac{q - 1}{\abs{\mathcal{G}}} - 1 \right) \cdot q^\frac{1}{2} \right).
        \end{align*}

        For $a = 0$, $b \neq 0$, by assumption $x \mapsto b \cdot f (x)$ is non-constant on $\mathcal{G}$.
        Therefore, $b \cdot f \left( x^\frac{q - 1}{\abs{\mathcal{G}}} \right) \not\equiv c \mod \left( x^q - x \right)$, where $c \in \Fq$.
        So, we apply \Cref{Lem: character sum over subgroup} and the Weil bound (\Cref{Th: Weil bound}) to
        \begin{align*}
            \abs{\mathcal{W}_f (\psi, \mathcal{G}, 0, b)}
            &= \frac{\abs{\mathcal{G}}}{q - 1} \cdot \abs{\sum_{x \in \Fqx} \psi \bigg( b \cdot f \Big( x^\frac{q - 1}{\abs{\mathcal{G}}} \Big) \bigg)} \\
            &= \frac{\abs{\mathcal{G}}}{q - 1} \cdot \abs{-\psi \big( f (0) \big) + \sum_{x \in \Fq} \psi \bigg( b \cdot f \Big( x^\frac{q - 1}{\abs{\mathcal{G}}} \Big) \bigg)} \\
            &\leq \frac{\abs{\mathcal{G}}}{q - 1} \cdot \left( 1 + \abs{\sum_{x \in \Fq} \psi \bigg( b \cdot f \Big( x^\frac{q - 1}{\abs{\mathcal{G}}} \Big) \bigg)} \right) \\
            &\leq \frac{\abs{\mathcal{G}}}{q - 1} \cdot \left(1 + \left( \degree{f} \cdot \frac{q - 1}{\abs{\mathcal{G}}} - 1 \right) \cdot q^\frac{1}{2} \right).
        \end{align*}

        For $a \cdot b \neq 0$, by assumption $x \mapsto a \cdot x + b \cdot f (x)$ is not constant on $\mathcal{G}$, so the character sum is non-trivial and the estimation is analogously to $a = 0$, $b \neq 0$.
    \end{proof}
    \begin{rem}
        If $f$ is a permutation polynomial which commutes with $x^\frac{q - 1}{\abs{\mathcal{G}}}$, then the estimation for $a = 0$, $b \neq 0$ improves to one of $a \neq 0$, $b = 0$.
        For example, this is the case if $f = x^d$ is a permutation monomial.
    \end{rem}

    If $\degree{f} \leq \abs{\mathcal{G}}$, then the technical condition of \Cref{Th: Walsh spectrum over subgroup} is always satisfied.
    \begin{lem}\label[lem]{Lem: non-constant on subgroup}
        Let $\Fq$ be a finite field, let $\mathcal{G} \subset \Fqx$ be a subgroup, and let $f \in \Fq [x]$ be such that $1 < \degree{f} \leq \abs{\mathcal{G}}$.
        Then for all $a \in \Fq$ and $b \in \Fqx$ the function $x \mapsto a \cdot x + b \cdot f (x)$ is non-constant on $\mathcal{G}$.
    \end{lem}
    \begin{proof}
        Analogously to \Cref{Lem: character sum over subgroup}, a polynomial representation of the function on the group $\mathcal{G}$ is $F (x) = a \cdot x^\frac{q - 1}{\abs{\mathcal{G}}} + b \cdot f \left( x^\frac{q - 1}{\abs{\mathcal{G}}} \right)$.
        Since $\degree{f} > 1$, we have $\degree{F} = \degree{f} \cdot \frac{q - 1}{\abs{\mathcal{G}}}$.
        By $\degree{f} \leq \abs{\mathcal{G}}$ also $\degree{F} < q$, so $F$ is already the unique representant of the function in $\Fq [x] / \left( x^q - x \right)$.
        In particular, this implies that $F \not\equiv c \mod \left( x^q - x \right)$, where $c \in \Fq$.
    \end{proof}

    Now we can again estimate the Walsh spectrum of power residue S-Boxes.
    \begin{cor}\label[cor]{Cor: Walsh spectrum power residue S-Box}
        Let $\Fq$ be a finite field, and let $m, d \in \Z_{> 1}$ be such that
        \begin{enumerate}[label=(\roman*)]
            \item $m \mid (q - 1)$,

            \item $d \cdot m < q$, and

            \item $\gcd \left( d, q \right) = 1$.
        \end{enumerate}
        Let $S: \Fq \to \Fq$ be the power residue S-Box of $(d, m, T)$, let $\psi: \Fq \to \C$ be a non-trivial additive character, and let $a, b \in \Fq$.
        Then
        \begin{equation*}
            \abs{\mathcal{W}_S (\psi, a, b)} \leq
            \begin{dcases}
                q, & a = b = 0, \\
                0, & a \neq 0, b = 0, \\
                \left( d \cdot m - 1 \right) \cdot q^\frac{1}{2} + 2, & a = 0, b \neq 0, \\
                \left( d \cdot m - 1 \right) \cdot q^\frac{1}{2} + 2, & a \cdot b \neq 0.
            \end{dcases}
        \end{equation*}
    \end{cor}
    \begin{proof}
        Analogously to the proof of \Cref{Cor: Walsh spectrum power residue S-Box inverse}, let $g \in \Fqx$ be a generator, and for $0 \leq r \leq m - 1$ let $N_r = \big\{ g^r, g^{m + r}, g^{2 \cdot m + r}, \ldots, g^{q - 1 - m + r} \big\}$.
        Then
        \begin{align*}
            \mathcal{W}_S (\psi, a, b)
            &= \sum_{x \in \Fq} \psi \big( a \cdot x + b \cdot S (x) \big) \\
            &= \psi (0) + \sum_{r = 0}^{m - 1} \sum_{x \in N_r} \psi \Bigg( a \cdot x + b \cdot x^d \cdot T \Bigg( \legendre{x}{q}_m \Bigg) \Bigg) \\
            &= \psi (0) + \sum_{r = 0}^{m - 1} \sum_{x \in N_r} \psi \Bigg( a \cdot x + b \cdot x^d \cdot T \bigg( g^{r \cdot \frac{q - 1}{m}} \bigg) \Bigg) \\
            &= \psi (0) + \sum_{r = 0}^{m - 1} \sum_{x \in N_0} \psi \Bigg( a \cdot x \cdot g^r + b \cdot x^d \cdot g^{d \cdot r} \cdot T \bigg( g^{r \cdot \frac{q - 1}{m}} \bigg) \Bigg).
        \end{align*}
        For $r$ fixed, we can consider the terms $g^r$ and $g^{d \cdot r} \cdot T \left( g^{r \cdot \frac{q - 1}{m}} \right)$ as non-zero constants.\footnote{Recall from \Cref{Def: power residue S-Box} that $0 \notin \imag \bigg( T \circ \legendre{x}{q}_m \bigg)$.}
        Therefore, when passing to the absolute value we yield with the triangular inequality that
        \begin{equation}\label{Equ: character sum estimation}
            \abs{\mathcal{W}_S (\psi, a, b)}
            \leq 1 + \sum_{r = 0}^{m - 1} \abs{\sum_{x \in N_0} \psi \left( a_r \cdot x + b_r \cdot x^d \right)}
            = 1 + \sum_{r = 0}^{m - 1} \abs{\mathcal{W}_{x^d} (\psi, N_0, a_r, b_r)},
        \end{equation}
        where $a_r = a \cdot g^r$ and $b_r = b \cdot g^{d \cdot r} \cdot T \left( g^{r \cdot \frac{q - 1}{m}} \right)$.
        Recall that $N_0$ is a group, namely $N_0 = \left< g^m \right> \subset \Fqx$.
        Also, by assumption we have that $d \cdot \frac{q - 1}{\abs{N_0}} = d \cdot m < q \Leftrightarrow d \leq \frac{q - 1}{m}$, so by \Cref{Lem: non-constant on subgroup} the function $x \mapsto a \cdot x + b \cdot f (x)$ is non-constant on $N_0$ for all $a \in \Fq$ and $b \in \Fqx$.
        Therefore, we can apply \Cref{Th: Walsh spectrum over subgroup} to estimate $\abs{\mathcal{W}_{x^d} (\psi, N_0, a_r, b_r)}$.

        Now we do a case distinction on $a, b \in \Fq$.
        \begin{itemize}
            \item The case $a = b = 0$ is trivial.

            \item The case $a \neq 0$, $b = 0$ is well-known \cite[5.31.~Theorem]{Lidl-FiniteFields}.

            \item For $a = 0$, $b \neq 0$ and $a \cdot b \neq 0$, we have with \Cref{Th: Walsh spectrum over subgroup} that
            \begin{equation*}
                \abs{\mathcal{W}_{x^d} (\psi, N_0, 0, b_r)}
                \leq \frac{1}{m} + \left( d - \frac{1}{m} \right) \cdot q^\frac{1}{2}. \qedhere
            \end{equation*}
        \end{itemize}
    \end{proof}

    \begin{rem}
        \begin{enumerate}
            \item If $S$ is a permutation, then the case $a = 0$, $b \neq 0$ improves to $\mathcal{W}_S (\psi, 0, b) = 0$.

            \item The bound from \Cref{Cor: Walsh spectrum power residue S-Box} is non-trivial if $d \cdot m < q^\frac{1}{2}$.
        \end{enumerate}
    \end{rem}

    \begin{ex}[\Grendel{} I]
        Let $\Fq$ be a prime field, let $d \in \Z_{> 1}$ be such that $2 \cdot d < q$, and let $S: \Fq \to \Fq$ be a $(d, 2, \id)$ power residue S-Box which induces a permutation of $\Fq$.
        I.e., $S (x) = x^d \cdot \legendre{x}{q}_2 = x^{d + \frac{q - 1}{2}}$ is a \Grendel{} S-Box \cite[\S 3]{EPRINT:Szepieniec21}.\footnote{Note that $S$ is a permutation if and only if $\gcd \left( d + \frac{q - 1}{2}, q - 1 \right)$ \cite[7.8.~Theorem]{Lidl-FiniteFields}.}
        Then the Walsh spectrum of $S$ is bounded by
        \begin{equation*}
            \abs{\mathcal{W}_S (\psi, a, b)} \leq
            \begin{dcases}
                q, & a = b = 0, \\
                0, & a \neq 0, b = 0,\ a = 0, b \neq 0 \\
                (2 \cdot d - 1) \cdot q^\frac{1}{2} + 2, & a \cdot b \neq 0.
            \end{dcases}
        \end{equation*}
        In particular, for $q > 3$ and any non-trivial additive characters $\psi, \phi: \Fq \to \C$ we have that
        \begin{equation*}
            \max_{a, b \in \Fqx} \abs{\CORR_S (\psi, \phi, a, b)} \leq \frac{2 \cdot d}{q^\frac{1}{2}}.
        \end{equation*}
    \end{ex}

    \subsection{\texorpdfstring{$\mathbf{d = 1}$}{d = 1}}\label{Sec: linear resiude S-Box}
    Recall from \Cref{Ex: d = 1} that the argument from \Cref{Cor: Walsh spectrum power residue S-Box} does not apply in full generality to $d = 1$.
    To settle this case satisfactorily we require more structural assumptions on the function $T$.
    \begin{cor}\label[cor]{Cor: Walsh spectrum power residue S-Box d = 1}
        Let $\Fq$ be a finite field, let $m \in \Z_{> 1}$ be such that $m \mid (q - 1)$, and let $S: \Fq \to \Fq$ be a $(1, m, T)$ power residue S-Box such that $T$ induces a permutation on the subgroup $N_0 \subset \Fqx$ of order $\frac{q - 1}{m}$.
        Let $\psi: \Fq \to \C$ be a non-trivial additive character, and let $a, b \in \Fq$.
        Then
        \begin{equation*}
            \abs{\mathcal{W}_S (\psi, a, b)} \leq
            \begin{dcases}
                q, & a = b = 0, \\
                0, & a \neq 0, b = 0, \\
                \left( m - 1 \right) \cdot q^\frac{1}{2} + 2, & a = 0, b \neq 0, \\
                \left( m - 1 \right) \cdot q^\frac{1}{2} + 2, & -\frac{a}{b} \notin N_0, \\
                \frac{q - 1}{m} + \left( m - 2 + \frac{1}{m} \right) \cdot q^\frac{1}{2} + 2 - \frac{1}{m}, & -\frac{a}{b} \in N_0.
            \end{dcases}
        \end{equation*}
    \end{cor}
    \begin{proof}
        Analogously to the proof of \Cref{Cor: Walsh spectrum power residue S-Box}, let $g \in \Fqx$ be a generator, and for $0 \leq r \leq m - 1$ let $N_r = \big\{ g^r, g^{m + r}, g^{2 \cdot m + r}, \ldots, g^{q - 1 - m + r} \big\}$.
        Then we have
        \begin{align}
            \mathcal{W}_S (\psi, a, b)
            &= \psi (0) + \sum_{r = 0}^{m - 1} \sum_{x \in N_0} \psi \Bigg( a \cdot x \cdot g^r + b \cdot x \cdot g^r \cdot T \bigg( g^{r \cdot \frac{q - 1}{m}} \bigg) \Bigg), \\
            \abs{\mathcal{W}_S (\psi, a, b)}
            &\leq 1 + \sum_{r = 0}^{m - 1} \abs{\mathcal{W}_x (\psi, N_0, a_r, b_r)}, \label{Equ: character sums linear monomial}
        \end{align}
        where $a_r = a \cdot g^r$ and $b_r = b \cdot g^r$.

        Now we do a case distinction on $a, b \in \Fq$.
        \begin{itemize}
            \item The case $a = b = 0$ is trivial.

            \item The case $a \neq 0$, $b = 0$ is well-known \cite[5.31.~Theorem]{Lidl-FiniteFields}.

            \item For $a = 0$, $b \neq 0$, with \Cref{Th: Walsh spectrum over subgroup} we have
            \begin{equation*}
                \abs{\mathcal{W}_x (\psi, N_0, 0, b_r)}
                \leq \frac{1}{m} + \left( 1 - \frac{1}{m} \right) \cdot q^\frac{1}{2}.
            \end{equation*}

            \item For $a \cdot b \neq 0$, let us first investigate when $x \mapsto a_r \cdot x + b_r \cdot x$ is equal to the zero function on $N_0$
            \begin{align*}
                a \cdot  g^r \cdot x + b \cdot g^r \cdot x \cdot T \Bigg(\legendre{g^r}{q}_m \Bigg) &\stackrel{!}{=} 0, \qquad \forall x \in N_0, \\
                \Longleftrightarrow T \left( g^{r \cdot \frac{q - 1}{m}} \right) &= -\frac{a}{b}.
            \end{align*}
            Since $T$ is assumed to be a permutation on $N_0$, we must have that $-\frac{a}{b} \in N_0$, and for every such pair there exists exactly one $0 \leq r \leq m - 1$ which satisfies the equality.

            Now we can do the following subcase distinction.
            \begin{itemize}
                \item For $-\frac{a}{b} \notin N_0$, all polynomial arguments in \Cref{Equ: character sums linear monomial} are non-constant on $N_0$.
                So, with \Cref{Th: Walsh spectrum over subgroup} we have
                \begin{equation}\label{Equ: linear character sum non-trivial case}
                    \abs{\mathcal{W}_x (\psi, N_0, a_r, b_r)}
                    \leq \frac{1}{m} + \left( 1 - \frac{1}{m} \right) \cdot q^\frac{1}{2}.
                \end{equation}

                \item For $-\frac{a}{b} \in N_0$, there exists exactly one $0 \leq r' \leq m - 1$ such that
                \begin{equation*}
                    \mathcal{W}_x (\psi, N_0, a_r, b_r) = \frac{q - 1}{m},
                \end{equation*}
                and for all $r \neq r'$ we have with \Cref{Th: Walsh spectrum over subgroup} the same bound as in \Cref{Equ: linear character sum non-trivial case}. \qedhere
            \end{itemize}
        \end{itemize}
    \end{proof}

    \begin{ex}[\Grendel{} II]
        Let $\Fq$ be a prime field, and let $S: \Fq \to \Fq$ be a $(1, 2, \id)$ power residue S-Box which induces a permutation of $\Fq$.
        I.e., $S (x) = x \cdot \legendre{x}{q}_2 = x^{\frac{q + 1}{2}}$ is a \Grendel{} S-Box \cite[\S 3]{EPRINT:Szepieniec21}.
        Then the Walsh spectrum of $S$ is bounded by
        \begin{equation*}
            \abs{\mathcal{W}_S (\psi, a, b)} \leq
            \begin{dcases}
                q, & a = b = 0, \\
                0, & a \neq 0, b = 0,\ a = 0, b \neq 0 \\
                q^\frac{1}{2} + 2, & \legendre{-a \cdot b}{q}_2 = -1, \\
                \frac{q^\frac{1}{2} + q}{2} + 1, & \legendre{-a \cdot b}{q}_2 = 1.
            \end{dcases}
        \end{equation*}
        In particular, for $q > 3$ and any non-trivial additive characters $\psi, \phi: \Fq \to \C$ we have that
        \begin{equation*}
            \max_{a, b \in \Fqx} \abs{\CORR_S (\psi, \phi, a, b)} \leq \frac{1}{2} + \frac{1}{2 \cdot q^\frac{1}{2}} + \frac{1}{q}.
        \end{equation*}
    \end{ex}

    \subsection{Another Legendre Symbol-Based S-Box}\label{Sec: Legendre symbol-based S-Box}
    Let $\Fq $ be a prime field with $q \geq 3$, and let $d_+, d_- \in \Z_{> 1}$.
    Grassi et al.\ \cite[\S 4.1]{ToSC:GKRS22} introduced a novel S-Box based on the Legendre symbol
    \begin{equation*}
        S (x) = \frac{x^{d_+} \cdot \bigg( 1 + \legendre{x}{q}_2 \bigg) + x^{d_-} \cdot \bigg( 1 - \legendre{x}{q}_2 \bigg)}{2} \in \Fq [x].
    \end{equation*}
    Note that $S$ induces a permutation if $\gcd \left( d_+ \cdot d_-, q - 1 \right) = 1$ \cite[Proposition~5]{ToSC:GKRS22}.

    With the results of \Cref{Sec: power permutation resiude S-Box} we can also obtain estimations on the Walsh spectrum of $S$.
    To abbreviate writing a bit, we use $d_\pm$ as placeholder if a property applies to $d_+$ as well as $d_-$.
    \begin{prop}\label[prop]{Prop: Walsh spectrum Legendre symbol based S-Box}
        Let $\Fq$ be a finite field with $q$ odd, let $\psi: \Fq \to \C$ be a non-trivial additive character, let $d_+, d_- \in \Z_{> 1}$ with $\gcd \left( d_+ \cdot d_-, q \right) = 1$ and $2 \cdot d_{\pm} < q$, let
        \begin{equation*}
            S (x) = \frac{x^{d_+} \cdot \bigg( 1 + \legendre{x}{q}_2 \bigg) + x^{d_-} \cdot \bigg( 1 - \legendre{x}{q}_2 \bigg)}{2} \in \Fq [x],
        \end{equation*}
        and let $a, b \in \Fq$.
        Then
        \begin{equation*}
            \abs{\mathcal{W}_S (\psi, a, b)} \leq
            \begin{dcases}
                q, & a = b = 0, \\
                0, & a \neq 0, b = 0, \\
                \left( d_+ + d_- - 1 \right) \cdot q^\frac{1}{2} + 2, & a = 0, b \neq 0, \\
                \left( d_+ + d_- - 1 \right) \cdot q^\frac{1}{2} + 2, & a \cdot b \neq 0.
            \end{dcases}
        \end{equation*}
    \end{prop}
    \begin{proof}
        Analogously to the proof of \Cref{Cor: Walsh spectrum power residue S-Box inverse}, let $g \in \Fqx$ be a generator, and for $r \in \{ 0, 1 \}$ let $N_r = \big\{ g^r, g^{2 + r}, g^{2 \cdot 2 + r}, \ldots, g^{q - 1 - 2 + r} \big\}$.
        Then
        \begin{align*}
            \mathcal{W}_S (\psi, a, b)
            &= \sum_{x \in \Fq} \psi \big( a \cdot x + b \cdot S (x) \big) \\
            &= \psi (0) + \sum_{r = 0}^{1} \sum_{x \in N_r} \psi \left( a \cdot x + b \cdot \frac{x^{d_+} \cdot \left( 1 + \legendre{x}{q}_2 \right) + x^{d_-} \cdot \left( 1 - \legendre{x}{q}_2 \right)}{2} \right).
        \end{align*}
        Note that for a generator $g \in \Fqx$ we must have that $\legendre{g}{q - 1}_2 = g^\frac{q - 1}{2} = -1$.
        So, for $x \in N_r$ we have that
        \begin{equation*}
            \legendre{x}{q}_2
            = \legendre{g^{k \cdot 2 + r}}{q}_2
            = g^{\left( k \cdot 2 + r \right) \cdot \frac{q - 1}{2}}
            = g^{r \cdot \frac{q - 1}{2}}
            =
            \begin{cases}
                1, & r = 0, \\
                -1, & r = 1.
            \end{cases}
        \end{equation*}
        Henceforth, the character sum simplifies to
        \begin{equation*}
            \mathcal{W}_S (\psi, a, b) = \psi (0) + \sum_{x \in N_0} \psi \left( a \cdot x + b \cdot x^{d_+} \right) + \psi \left( a \cdot g \cdot x + b \cdot g^{d_-} \cdot x^{d_-} \right).
        \end{equation*}
        Therefore, when passing to the absolute value we yield with the triangular inequality that
        \begin{equation*}
            \abs{\mathcal{W}_S (\psi, a, b)}
            = 1 + \abs{\mathcal{W}_{x^{d_+}} (\psi, N_0, a, b)} + \abs{\mathcal{W}_{x^{d_-}} \Big( \psi, N_0, a \cdot g, b \cdot g^{d_-} \Big)}.
        \end{equation*}
        Also, note that $N_0$ is a group, namely the subgroup of squares
        \begin{equation*}
            N_0 = \left\{ x \in \Fqx \; \middle\vert \; \legendre{x}{q}_2 = 1  \right\} \subset \Fqx.
        \end{equation*}
        (The identity of these two groups follows again from the uniqueness of subgroups in cyclic groups \cite[1.15.~Theorem]{Lidl-FiniteFields}.)
        By assumption $2 \cdot d_{\pm} < q$, so by \Cref{Lem: non-constant on subgroup} the function $x \mapsto a \cdot x + b \cdot x^{d_\pm}$ is non-constant on $N_0$ for all $a \in \Fq$ and $b \in \Fqx$.
        Hence, we can apply \Cref{Th: Walsh spectrum over subgroup} to estimate $\abs{\mathcal{W}_{x^{d_\pm}} (\psi, N_0, a_r, b_r)}$, where $a_r = a \cdot g^r$ and $b_r = b \cdot g^{d_- \cdot r}$.

        Now we do a case distinction on $a, b \in \Fq$.
        \begin{itemize}
            \item The case $a = b = 0$ is trivial.

            \item The case $a \neq 0$, $b = 0$ is well-known \cite[5.31.~Theorem]{Lidl-FiniteFields}.

            \item The cases $a = 0$, $b \neq 0$ and $a \cdot b \neq 0$ follow with \Cref{Th: Walsh spectrum over subgroup}
            \begin{equation*}
                \abs{\mathcal{W}_{x^{d_\pm}} (\psi, N_0, 0, b_r)}
                \leq \frac{1}{2} + \left( d_\pm - \frac{1}{2} \right) \cdot q^\frac{1}{2}. \qedhere
            \end{equation*}
        \end{itemize}
    \end{proof}

    \begin{rem}
        If in addition $\gcd \left( d_+ \cdot d_-, q - 1 \right) = 1$, then $S$ induces a permutation \cite[Proposition~5]{ToSC:GKRS22} and the case $a = 0$, $b \neq 0$ improves to $\mathcal{W}_S (\psi, 0, b) = 0$.
    \end{rem}

    In particular, for $q > 5$ and any non-trivial additive characters $\psi, \phi: \Fq \to \C$ \Cref{Prop: Walsh spectrum Legendre symbol based S-Box} implies that
    \begin{equation*}
        \max_{a, b \in \Fqx} \abs{\CORR_S (\psi, \phi, a, b)} \leq \frac{d_+ + d_-}{q^\frac{1}{2}}.
    \end{equation*}

    \section{Numerical Computations}
    We have implemented character sums with the computer algebra program \cite{OSCAR}.
    Our code is publicly available on GitHub.\footnote{\url{https://github.com/SteinerMatthias/Walsh-Spectrum-of-Power-Residue-S-Boxes}}
    For small primes we have exhaustively computed the various Kloosterman and Walsh spectra of this paper to observe our estimations in practice.

    \subsection{Kloosterman Sums Over Subgroups}
    Let us start with Kloosterman sums over subgroups.
    As always, let $a, b, g \in \Fqx$ be such that $g$ is a cyclic generator, and let $m \in \Z_{> 1}$ be such that $m \mid (q - 1)$.
    Then $\left< g^m \right> \subset \Fqx$ is a subgroup of order $\frac{q - 1}{m}$.
    In addition, we have that $a = g^{k \cdot m + r}$ for some $k, r \in \Z_{\geq 0}$ such that $0 \leq r \leq m - 1$.
    For Kloosterman sums over $\left< g^m \right>$ we have the following identity
    \begin{align}
        K \left(\psi, \left< g^m \right>, a, b \right)
        &= \sum_{x \in \left< g^m \right>} \psi \left( a \cdot x + b \cdot x^{-1} \right) \\
        &= \sum_{x \in \left< g^m \right>} \psi \left( g^{k \cdot m + r} \cdot x + b \cdot x^{-1} \right) \\
        &= \sum_{y \in \left< g^m \right>} \psi \left( g^r \cdot y + \frac{b}{g^{k \cdot m}} \cdot y^{-1} \right) \label{Equ: Kloosterman substituion} \\
        &= K \left(\psi, \left< g^m \right>, g^r, \frac{b}{g^{k \cdot m}} \right),
    \end{align}
    where \Cref{Equ: Kloosterman substituion} follows from the fact that $x \mapsto g^{k \cdot m} \cdot x = y$ is a bijection on $x \in \left< g^m \right>$.
    Therefore, to compute the full Kloosterman spectrum with respect to a fixed generator $g$ it is sufficient to compute the spectrum over $(a, b) \in \left\{ 1, g, \dots, g^{r - 1} \right\} \times \Fqx$.
    I.e., we only have to evaluate $m \cdot (q - 1)$ sums rather than the generic $(q - 1)^2$.

    In \Cref{Fig: Kloosterman sum} we have plotted the maximal value of Kloosterman sums over $\left< g^m \right>$ with $m \in \{ 2, 4, 8, 16 \}$ and primes $3 \leq p \leq 2048$.
    For $m = 2$, our bound from \Cref{Th: Walsh spectrum over subgroup} indeed becomes tight.
    Though, for increasing $m$ the distance between the proven bound and our data points seems to increase.
    \begin{figure}[H]
        \centering
        \begin{tikzpicture}[scale=0.8]
            \begin{axis}[
                xmin = 0, xmax = 2048,
                ymin = 0, ymax = 2.5,
                xtick distance = 500,
                ytick distance = 0.5,
                minor x tick num=4,
                minor y tick num=4,
                grid = both,
                major grid style = {lightgray},
                minor grid style = {lightgray!25},
                xlabel = {$p$},
                ylabel = {$\max_{a, b \in \mathbb{F}_p^\times} \frac{\left| K \left( \psi, \left< g^m \right>, a, b \right) \right|}{\sqrt{p}}$},
                legend style={at={(axis cs:2094,2.5)},anchor=north west},
                legend cell align = {left},
                ]

                % m = 4
                \addplot[
                black,
                only marks,
                mark=x,
                mark size=1,
                ] table[skip first n=8, col sep=tab]
                {./Figure-Data/Kloosterman_m_2.log};

                % m = 4
                \addplot[
                blue,
                only marks,
                mark=o,
                mark size=1,
                ] table[skip first n=8, col sep=tab]
                {./Figure-Data/Kloosterman_m_4.log};

                % m = 8
                \addplot[
                olive,
                only marks,
                mark=square,
                mark size=1,
                ] table[skip first n=8, col sep=tab]
                {./Figure-Data/Kloosterman_m_8.log};

                % m = 16
                \addplot[
                red,
                only marks,
                mark=diamond,
                mark size=2,
                ] table[skip first n=8, col sep=tab]
                {./Figure-Data/Kloosterman_m_16.log};

                % Add legend
                \addlegendentry{$m = 2$}
                \addlegendentry{$m = 4$}
                \addlegendentry{$m = 8$}
                \addlegendentry{$m = 16$}
            \end{axis}
        \end{tikzpicture}
        \caption{Maximal Kloosterman sum over subgroups $\left< g^m \right> \subset \Fpx$ and primes $3 \leq p \leq 2048$.}
        \label{Fig: Kloosterman sum}
    \end{figure}

    \subsection{Power Residue S-Boxes With \texorpdfstring{$\mathbf{d = q - 2}$}{d = q - 2}}
    Next let us consider \Polocolo{} inspired S-Boxes $S (x) = x^{q - 2} \cdot \legendre{x}{q}_m$.
    To enumerate the spectrum efficiently we can exploit the following identity.
    \begin{lem}\label[lem]{Lem: efficient spectrum}
        Let $\Fq$ be a finite field, let $d, m \in \Z_{> 1}$ be such that $m \mid (q - 1)$, let $S (x) = x^d \cdot \legendre{x}{q}_m$, and let $a \in \Fqx$ and $b \in \Fq$.
        Then
        \begin{equation*}
            W_S (\psi, a, b)
            = W_S \Bigg( \psi, 1, \frac{b}{a^d} \cdot \legendre{a^{-1}}{q}_m \Bigg).
        \end{equation*}
    \end{lem}
    \begin{proof}
        A simple substitution yields
        \begin{align*}
            W_S (\psi, a, b)
            &= \sum_{x \in \Fq} \psi \Bigg( a \cdot x + b \cdot x^d \cdot \legendre{x}{q}_m \Bigg) \\
            &= \sum_{y \in \Fq} \psi \Bigg( y + \frac{b}{a^d} \cdot y^d \cdot \legendre{a^{-1}}{q}_m \cdot \legendre{y}{q}_m \Bigg) \\
            &= W_S \Bigg( \psi, 1, \frac{b}{a^d} \cdot \legendre{a^{-1}}{q}_m \Bigg). \qedhere
        \end{align*}
    \end{proof}

    I.e., for the Walsh spectrum of $S$ with $a \neq 0$ we only have to evaluate $(q - 1)$ sums rather than the generic $(q - 1)^2$.

    In \Cref{Fig: Inverse} we have plotted the maximal Walsh spectrum of $S$ with $m \in \{ 2, 4, 8, 16 \}$ and primes $3 \leq p \leq 2048$.
    Recall that our bound from \Cref{Cor: Walsh spectrum power residue S-Box inverse} is $2 \cdot m \cdot p^\frac{1}{2}$.
    For $m = 2$ the bound is almost matched for some small primes.
    Though, for increasing $m$ the distance between the proven bound and our data points seems to increase.

    \clearpage % To fix box issue
    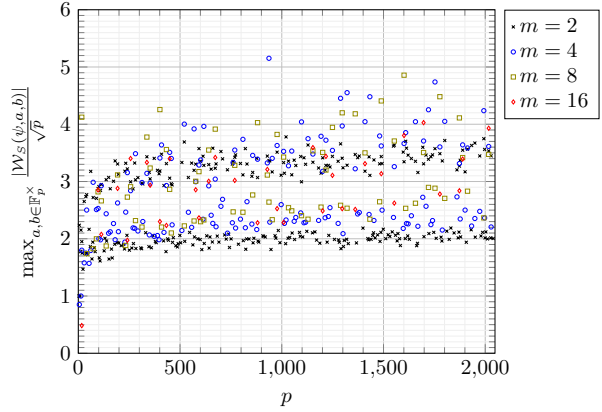
\begin{figure}[H]
        \centering
        \begin{tikzpicture}[scale=0.8]
            \begin{axis}[
                xmin = 0, xmax = 2048,
                ymin = 0, ymax = 6,
                xtick distance = 500,
                ytick distance = 1,
                minor x tick num=4,
                minor y tick num=9,
                grid = both,
                major grid style = {lightgray},
                minor grid style = {lightgray!25},
                xlabel = {$p$},
                ylabel = {$\max_{a, b \in \mathbb{F}_p^\times} \frac{\left| \mathcal{W}_S \left( \psi, a, b \right) \right|}{\sqrt{p}}$},
                legend style={at={(axis cs:2094,6)},anchor=north west},
                legend cell align = {left},
                ]

                % m = 2
                \addplot[
                black,
                only marks,
                mark=x,
                mark size=1,
                ] table[skip first n=8, col sep=tab]
                {./Figure-Data/PR_Sbox_Inverse_m_2.log};

                % m = 4
                \addplot[
                blue,
                only marks,
                mark=o,
                mark size=1,
                ] table[skip first n=8, col sep=tab]
                {./Figure-Data/PR_Sbox_Inverse_m_4.log};

                % m = 8
                \addplot[
                olive,
                only marks,
                mark=square,
                mark size=1,
                ] table[skip first n=8, col sep=tab]
                {./Figure-Data/PR_Sbox_Inverse_m_8.log};

                % m = 16
                \addplot[
                red,
                only marks,
                mark=diamond,
                mark size=1,
                ] table[skip first n=8, col sep=tab]
                {./Figure-Data/PR_Sbox_Inverse_m_16.log};

                % Add legend
                \addlegendentry{$m = 2$}
                \addlegendentry{$m = 4$}
                \addlegendentry{$m = 8$}
                \addlegendentry{$m = 16$}
            \end{axis}
        \end{tikzpicture}
        \caption{Maximal Walsh spectrum for $S (x) = x^{q - 2} \cdot \legendre{x}{q}_m$ and primes $3 \leq p \leq 2048$.}
        \label{Fig: Inverse}
    \end{figure}

    \subsection{Power Residue S-Boxes With Small \texorpdfstring{$\mathbf{d}$}{d}}
    Next let us consider S-Boxes of the form $S (x) = x^d \cdot \legendre{x}{q}_m$ with $d \ll q^\frac{1}{2}$.
    Again, we can use \Cref{Lem: efficient spectrum} to compute the spectrum efficiently.

    In \Cref{Fig: 1 <= d <= 5} we have plotted the maximal value of the Walsh spectrum with $1 \leq d \leq 5$,  $m \in \{ 2, 4, 8, 16 \}$ and primes $3 \leq p \leq 2048$.
    Recall that for $d > 1$ our bound from \Cref{Cor: Walsh spectrum power residue S-Box} is $(d \cdot m - 1) \cdot p^\frac{1}{2}$, and for none of our experiments it becomes tight.
    Also, for increasing $d$ and $m$ we observe an increasing distance between the proven bound and the data points.
    On the other hand, for $d = 1$ and $m = 2, 4$ our theoretical bound from \Cref{Cor: Walsh spectrum power residue S-Box d = 1} is matched, but for larger $m$ there is again an increasing gap between the theoretical bound and the empirical values.

    %\clearpage % To fix box issue
    \subsection{Grassi et al.'s S-Box}
    For the Legendre symbol-based S-Box of Grassi et al.\ (\Cref{Sec: Legendre symbol-based S-Box}) we have computed the maximal Walsh spectrum for $(d_+, d_-) = (3, 5)$, $(d_+, d_-) = (5, 7)$ and primes $3 \leq p \leq 1024$.
    The results are depicted in \Cref{Fig: GKRS}.
    Our theoretical bounds from \Cref{Prop: Walsh spectrum Legendre symbol based S-Box} evaluate to $7 \cdot p^\frac{1}{2}$ and $11 \cdot p^\frac{1}{2}$ respectively.
    Again, they are not tightly matched.

%    \clearpage % To fix overfull box issue
    \begin{figure}[H]
        \centering
        \begin{tikzpicture}[scale=0.8]
            \begin{axis}[
                domain=1:2048,
                xmin = 0, xmax = 1024,
                ymin = 1, ymax = 6,
                xtick distance = 200,
                ytick distance = 1,
                minor x tick num=3,
                minor y tick num=9,
                grid = both,
                major grid style = {lightgray},
                minor grid style = {lightgray!25},
                xlabel = {$p$},
                ylabel = {$\max_{a, b \in \mathbb{F}_p^\times} \frac{\left| \mathcal{W}_S \left( \psi, a, b \right) \right|}{\sqrt{p}}$},
                legend style={at={(axis cs:1000,1.1)},anchor=south east},
                legend cell align = {left},
                ]

                % m = 2
                \addplot[
                black,
                only marks,
                mark=x,
                mark size=1,
                ] table[skip first n=9, col sep=tab]
                {./Figure-Data/GKRS_Sbox_d_+_3_d_-_5.log};

                % m = 4
                \addplot[
                blue,
                only marks,
                mark=o,
                mark size=1,
                ] table[skip first n=9, col sep=tab]
                {./Figure-Data/GKRS_Sbox_d_+_5_d_-_7.log};

                % Add legend
                \addlegendentry{$d_+ = 3$, $d_- = 5$}
                \addlegendentry{$d_+ = 5$, $d_- = 7$}
            \end{axis}
        \end{tikzpicture}
        \caption{Maximal Walsh spectrum for $S (x) = \frac{x^{d_+} \cdot \Big( 1 + \legendre{x}{q}_2 \Big) + x^{d_-} \cdot \Big( 1 - \legendre{x}{q}_2 \Big)}{2}$ and primes $3 \leq p \leq 1024$.}
        \label{Fig: GKRS}
    \end{figure}

    \clearpage % To fix overfull box issue
    \begin{figure}[H]
        \centering
        \begin{subfigure}{0.99\textwidth}
            \centering
            \begin{tikzpicture}[scale=0.75, declare function={ f (\q,\m) = sqrt(\q) / \m + \m - 2 + 1 / \m; }]
                \begin{axis}[
                    domain=1:2048,
                    xmin = 0, xmax = 2048,
                    ymin = 0.5, ymax = 25,
                    xtick distance = 500,
                    ytick distance = 5,
                    minor x tick num=4,
                    minor y tick num=4,
                    grid = both,
                    major grid style = {lightgray},
                    minor grid style = {lightgray!25},
                    xlabel = {$p$},
                    ylabel = {$\max_{a, b \in \mathbb{F}_p^\times} \frac{\left| \mathcal{W}_S \left( \psi, a, b \right) \right|}{\sqrt{p}}$},
                    legend style={at={(axis cs:2094,25)},anchor=north west},
                    legend cell align = {left},
                    ]

                    % m = 2
                    \addplot[
                    black,
                    only marks,
                    mark=x,
                    mark size=1,
                    ] table[skip first n=9, col sep=tab]
                    {./Figure-Data/PR_Sbox_d_1_m_2.log};

                    % m = 4
                    \addplot[
                    blue,
                    only marks,
                    mark=o,
                    mark size=1,
                    ] table[skip first n=9, col sep=tab]
                    {./Figure-Data/PR_Sbox_d_1_m_4.log};

                    % m = 8
                    \addplot[
                    olive,
                    only marks,
                    mark=square,
                    mark size=1,
                    ] table[skip first n=9, col sep=tab]
                    {./Figure-Data/PR_Sbox_d_1_m_8.log};

                    % m = 16
                    \addplot[
                    red,
                    only marks,
                    mark=diamond,
                    mark size=1,
                    ] table[skip first n=9, col sep=tab]
                    {./Figure-Data/PR_Sbox_d_1_m_16.log};

                    \addplot [black,
                    smooth] {
                        f(x, 2)
                    };

                    \addplot [blue,
                    smooth] {
                        f(x, 4)
                    };

                    \addplot [olive,
                    smooth] {
                        f(x, 8)
                    };

                    \addplot [red,
                    smooth] {
                        f(x, 16)
                    };

                    % Add legend
                    \addlegendentry{$m = 2$}
                    \addlegendentry{$m = 4$}
                    \addlegendentry{$m = 8$}
                    \addlegendentry{$m = 16$}
                    \addlegendentry{$\frac{\sqrt{p}}{2} + \frac{1}{2}$}
                    \addlegendentry{$\frac{\sqrt{p}}{4} + \frac{9}{4}$}
                    \addlegendentry{$\frac{\sqrt{p}}{8} + \frac{49}{8}$}
                    \addlegendentry{$\frac{\sqrt{p}}{16} + \frac{225}{16}$}
                \end{axis}
            \end{tikzpicture}
            \caption{$d = 1$.}
        \end{subfigure}
        \begin{subfigure}{0.49\textwidth}
            \centering
            \resizebox{\columnwidth}{!}{
                \begin{tikzpicture}
                    \begin{axis}[
                        xmin = 0, xmax = 2048,
                        ymin = 0.5, ymax = 4.5,
                        xtick distance = 500,
                        ytick distance = 0.5,
                        minor x tick num=4,
                        minor y tick num=4,
                        grid = both,
                        major grid style = {lightgray},
                        minor grid style = {lightgray!25},
                        xlabel = {$p$},
                        ylabel = {$\max_{a, b \in \mathbb{F}_p^\times} \frac{\left| \mathcal{W}_S \left( \psi, a, b \right) \right|}{\sqrt{p}}$},
                        legend style={at={(axis cs:2002,0.6)},anchor=south east},
                        legend cell align = {left},
                        ]

                        % m = 2
                        \addplot[
                        black,
                        only marks,
                        mark=x,
                        mark size=1,
                        ] table[skip first n=9, col sep=tab]
                        {./Figure-Data/PR_Sbox_d_2_m_2.log};

                        % m = 4
                        \addplot[
                        blue,
                        only marks,
                        mark=o,
                        mark size=1,
                        ] table[skip first n=9, col sep=tab]
                        {./Figure-Data/PR_Sbox_d_2_m_4.log};

                        % m = 8
                        \addplot[
                        olive,
                        only marks,
                        mark=square,
                        mark size=1,
                        ] table[skip first n=9, col sep=tab]
                        {./Figure-Data/PR_Sbox_d_2_m_8.log};

                        % m = 16
                        \addplot[
                        red,
                        only marks,
                        mark=diamond,
                        mark size=1,
                        ] table[skip first n=9, col sep=tab]
                        {./Figure-Data/PR_Sbox_d_2_m_16.log};

                        % Add legend
                        \addlegendentry{$m = 2$}
                        \addlegendentry{$m = 4$}
                        \addlegendentry{$m = 8$}
                        \addlegendentry{$m = 16$}
                    \end{axis}
                \end{tikzpicture}
            }
            \caption{$d = 2$}
            \label{Fig: d = 2}
        \end{subfigure}
        \hfill
        \begin{subfigure}{0.49\textwidth}
            \centering
            \resizebox{\columnwidth}{!}{
                \begin{tikzpicture}
                    \begin{axis}[
                        xmin = 0, xmax = 2048,
                        ymin = 0.5, ymax = 5,
                        xtick distance = 500,
                        ytick distance = 0.5,
                        minor x tick num=4,
                        minor y tick num=4,
                        grid = both,
                        major grid style = {lightgray},
                        minor grid style = {lightgray!25},
                        xlabel = {$p$},
                        ylabel = {$\max_{a, b \in \mathbb{F}_p^\times} \frac{\left| \mathcal{W}_S \left( \psi, a, b \right) \right|}{\sqrt{p}}$},
                        legend style={at={(axis cs:2002,0.6)},anchor=south east},
                        legend cell align = {left},
                        ]

                        % m = 2
                        \addplot[
                        black,
                        only marks,
                        mark=x,
                        mark size=1,
                        ] table[skip first n=9, col sep=tab]
                        {./Figure-Data/PR_Sbox_d_3_m_2.log};

                        % m = 4
                        \addplot[
                        blue,
                        only marks,
                        mark=o,
                        mark size=1,
                        ] table[skip first n=9, col sep=tab]
                        {./Figure-Data/PR_Sbox_d_3_m_4.log};

                        % m = 8
                        \addplot[
                        olive,
                        only marks,
                        mark=square,
                        mark size=1,
                        ] table[skip first n=9, col sep=tab]
                        {./Figure-Data/PR_Sbox_d_3_m_8.log};

                        % m = 16
                        \addplot[
                        red,
                        only marks,
                        mark=diamond,
                        mark size=1,
                        ] table[skip first n=9, col sep=tab]
                        {./Figure-Data/PR_Sbox_d_3_m_16.log};

                        % Add legend
                        \addlegendentry{$m = 2$}
                        \addlegendentry{$m = 4$}
                        \addlegendentry{$m = 8$}
                        \addlegendentry{$m = 16$}
                    \end{axis}
                \end{tikzpicture}
            }
            \caption{$d = 3$}
            \label{Fig: d = 3}
        \end{subfigure}
        \begin{subfigure}{0.49\textwidth}
            \centering
            \resizebox{\columnwidth}{!}{
                \begin{tikzpicture}
                    \begin{axis}[
                        xmin = 0, xmax = 2048,
                        ymin = 0.5, ymax = 5,
                        xtick distance = 500,
                        ytick distance = 0.5,
                        minor x tick num=4,
                        minor y tick num=4,
                        grid = both,
                        major grid style = {lightgray},
                        minor grid style = {lightgray!25},
                        xlabel = {$p$},
                        ylabel = {$\max_{a, b \in \mathbb{F}_p^\times} \frac{\left| \mathcal{W}_S \left( \psi, a, b \right) \right|}{\sqrt{p}}$},
                        legend style={at={(axis cs:2002,0.6)},anchor=south east},
                        legend cell align = {left},
                        ]

                        % m = 2
                        \addplot[
                        black,
                        only marks,
                        mark=x,
                        mark size=1,
                        ] table[skip first n=9, col sep=tab]
                        {./Figure-Data/PR_Sbox_d_4_m_2.log};

                        % m = 4
                        \addplot[
                        blue,
                        only marks,
                        mark=o,
                        mark size=1,
                        ] table[skip first n=9, col sep=tab]
                        {./Figure-Data/PR_Sbox_d_4_m_4.log};

                        % m = 8
                        \addplot[
                        olive,
                        only marks,
                        mark=square,
                        mark size=1,
                        ] table[skip first n=9, col sep=tab]
                        {./Figure-Data/PR_Sbox_d_4_m_8.log};

                        % m = 16
                        \addplot[
                        red,
                        only marks,
                        mark=diamond,
                        mark size=1,
                        ] table[skip first n=9, col sep=tab]
                        {./Figure-Data/PR_Sbox_d_4_m_16.log};

                        % Add legend
                        \addlegendentry{$m = 2$}
                        \addlegendentry{$m = 4$}
                        \addlegendentry{$m = 8$}
                        \addlegendentry{$m = 16$}
                    \end{axis}
                \end{tikzpicture}
            }
            \caption{$d = 4$}
            \label{Fig: d = 4}
        \end{subfigure}
        \hfill
        \begin{subfigure}{0.49\textwidth}
            \centering
            \resizebox{\columnwidth}{!}{
                \begin{tikzpicture}
                    \begin{axis}[
                        xmin = 0, xmax = 2048,
                        ymin = 0.5, ymax = 5,
                        xtick distance = 500,
                        ytick distance = 0.5,
                        minor x tick num=4,
                        minor y tick num=4,
                        grid = both,
                        major grid style = {lightgray},
                        minor grid style = {lightgray!25},
                        xlabel = {$p$},
                        ylabel = {$\max_{a, b \in \mathbb{F}_p^\times} \frac{\left| \mathcal{W}_S \left( \psi, a, b \right) \right|}{\sqrt{p}}$},
                        legend style={at={(axis cs:2002,0.6)},anchor=south east},
                        legend cell align = {left},
                        ]

                        % m = 2
                        \addplot[
                        black,
                        only marks,
                        mark=x,
                        mark size=1,
                        ] table[skip first n=9, col sep=tab]
                        {./Figure-Data/PR_Sbox_d_5_m_2.log};

                        % m = 4
                        \addplot[
                        blue,
                        only marks,
                        mark=o,
                        mark size=1,
                        ] table[skip first n=9, col sep=tab]
                        {./Figure-Data/PR_Sbox_d_5_m_4.log};

                        % m = 8
                        \addplot[
                        olive,
                        only marks,
                        mark=square,
                        mark size=1,
                        ] table[skip first n=9, col sep=tab]
                        {./Figure-Data/PR_Sbox_d_5_m_8.log};

                        % m = 16
                        \addplot[
                        red,
                        only marks,
                        mark=diamond,
                        mark size=1,
                        ] table[skip first n=9, col sep=tab]
                        {./Figure-Data/PR_Sbox_d_5_m_16.log};

                        % Add legend
                        \addlegendentry{$m = 2$}
                        \addlegendentry{$m = 4$}
                        \addlegendentry{$m = 8$}
                        \addlegendentry{$m = 16$}
                    \end{axis}
                \end{tikzpicture}
            }
            \caption{$d = 5$}
            \label{Fig: d = 5}
        \end{subfigure}
        \caption{Maximal Walsh spectrum for power residue S-Boxes $S (x) = x^d \cdot \legendre{x}{p}_m$ and primes $3 \leq p \leq 2048$.}
        \label{Fig: 1 <= d <= 5}
    \end{figure}
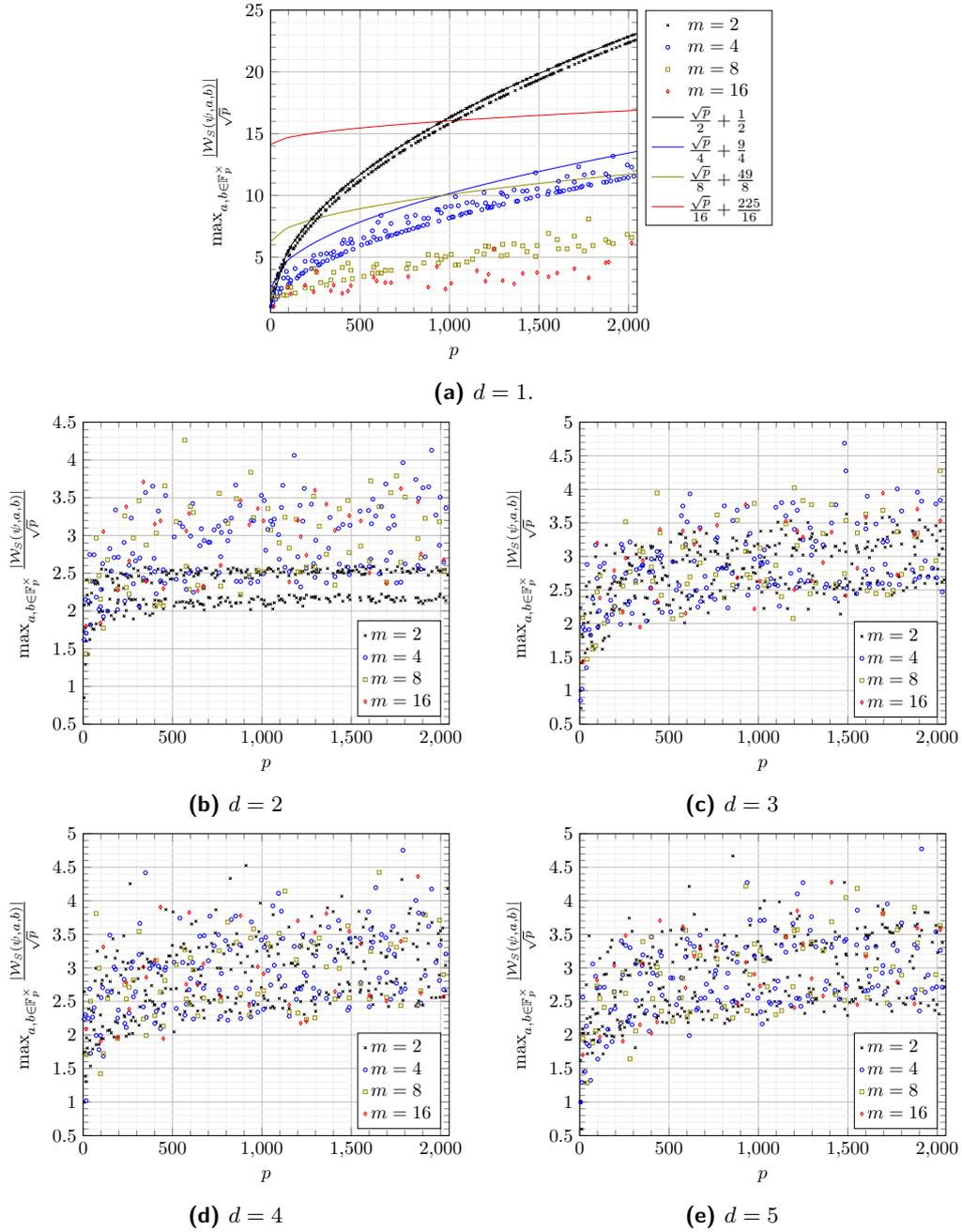

    \section{Discussion}
    In this paper we estimated the Walsh spectrum of power residue S-Boxes.
    For non-triviality of our estimations, in \Cref{Th: Kloosterman spectrum over subgroup,Th: Walsh spectrum over subgroup} we implicitly assumed that the subgroup $\mathcal{G} \subset \Fqx$ is large, i.e.\ $\abs{\mathcal{G}} > q^\frac{1}{2}$.
    Otherwise, we could always bound the character sums by $\abs{\mathcal{G}}$.
    Equivalently, in terms of \Cref{Cor: Walsh spectrum power residue S-Box inverse}, \ref{Cor: Walsh spectrum power residue S-Box} and \Cref{Prop: Walsh spectrum Legendre symbol based S-Box} we require that the divisor $m \mid (q - 1)$ is small enough so that $\frac{q - 1}{m} > q^\frac{1}{2}$.

    In case one has to estimate the Walsh spectrum of power residue S-Boxes with $\frac{q - 1}{m} < q^\frac{1}{2}$ one could try to incorporate the estimations from \cite{Ostafe-Equations,Ostafe-Weil} to improve upon the trivial bound.
    However, at least in the context of Zero-Knowledge-friendly primitives we are not aware of any power residue S-Box with $\frac{q - 1}{m} < q^\frac{1}{2}$.

    We also note that the estimations of \Cref{Th: Walsh spectrum over subgroup} and \Cref{Cor: Walsh spectrum power residue S-Box} can be straightforwardly generalized to S-Boxes of the form
    \begin{equation}
        S (x) = f (x) \cdot T \Bigg( \legendre{x}{q}_m \Bigg),
    \end{equation}
    where $f \in \Fq [x]$ is an arbitrary polynomial with $\gcd  \big( \degree{f} , q \big) = 1$ and $1 < \degree{f} \cdot m < q$.
    For this case we have
    \begin{align}
        \sum_{x \in N_r} \psi \Bigg( a \cdot x + b \cdot f (x) \cdot T \Bigg( \legendre{x}{q}_m \Bigg) \Bigg)
        &= \sum_{x \in N_r} \psi \Bigg( a \cdot x + b \cdot f (x) \cdot T \bigg( g^{r \cdot \frac{q - 1}{m}} \bigg) \Bigg) \\
        &= \sum_{x \in N_0} \psi \Bigg( a \cdot x \cdot g^r + b \cdot f (x \cdot g^r) \cdot T \bigg( g^{r \cdot \frac{q - 1}{m}} \bigg) \Bigg) \\
        &= \sum_{x \in N_0} \psi \Bigg( a \cdot x \cdot g^r + b \cdot \tilde{f} (x) \cdot T \bigg( g^{r \cdot \frac{q - 1}{m}} \bigg) \Bigg),
    \end{align}
    where $\tilde{f} (x) = f (x \cdot g^r) \in \Fq [x]$ is again of degree $\degree{f}$ for any $g^r$, and all our estimations can be extended.

    Finally, let us quickly showcase how our estimations can be applied to estimate the correlation of \Polocolo{} linear trails.
    Let $\mathbf{M} \in \Fq^{t \times t}$ be an invertible matrix, let $\mathbf{c} \in \Fq^t$ be a constant, and let $S: \Fq \to \Fq$ be a \Polocolo{} S-Box \cite[\S 3]{EC:HHLPS25} with power residue $m = 2^n < 2^{-2} \cdot q^\frac{1}{2}$.
    A \Polocolo{} round function is then simply a SPN
    \begin{equation}
        \mathcal{R} (\mathbf{x}) = \mathbf{M}
        \begin{pmatrix}
            S (x_1) \\ \vdots \\ S (x_n)
        \end{pmatrix}
        + \mathbf{c}.
    \end{equation}
    Now let $\mathbf{a}, \mathbf{b} \in \Fq^t \setminus \{ \mathbf{0} \}$ be linear masks, then the Walsh spectrum of the linear approximation $(\mathbf{a}, -\mathbf{b})$ is given by
    \begin{align}
        \mathcal{W}_\mathcal{R} \left( \psi, \mathbf{a}, -\mathbf{b} \right)
        &= \sum_{\mathbf{x} \in \Fq^t} \psi \big( \mathbf{a}^\intercal \mathbf{x} - \mathbf{b}^\intercal \mathcal{R} (\mathbf{x}) \big) \\
        &= \psi \left( \mathbf{b}^\intercal \mathbf{c} \right) \cdot \sum_{\mathbf{x} \in \Fq^t} \psi \left( \mathbf{a}^\intercal \mathbf{x} - \sum_{i = 1}^{t} \left( \mathbf{M}^\intercal \mathbf{b} \right)_i \cdot S (x_i) \right) \\
        &= \psi \left( \mathbf{b}^\intercal \mathbf{c} \right) \cdot \sum_{\mathbf{x} \in \Fq^t} \prod_{i = 1}^{t} \psi \big( a_i \cdot x_i - \left( \mathbf{M}^\intercal \mathbf{b} \right)_i \cdot S (x_i) \big) \\
        &= \psi \left( \mathbf{b}^\intercal \mathbf{c} \right) \cdot \prod_{i = 1}^{t} \mathcal{W}_S \big( \psi, a_i, - \left( \mathbf{M}^\intercal \mathbf{b} \right)_i \big).
    \end{align}
    Since $S$ is assumed to be a permutation this expression is equal to zero whenever $a_i \cdot \left( \mathbf{M}^\intercal \mathbf{b} \right)_i = 0$ for some $i$.
    Thus, with \Cref{Ex: Polocolo} we can now bound the correlation of a \Polocolo{} round for $\mathbf{a}, \mathbf{b} \neq \mathbf{0}$
    \begin{equation}
        \abs{\CORR_\mathcal{R} \left( \psi, \mathbf{a}, \mathbf{b} \right)} \leq
        \begin{dcases}
            0, & \exists i \! : a_i \cdot \left( \mathbf{M}^\intercal \mathbf{b} \right)_i = 0, \\
            \left( \frac{2^{n + 2}}{q^\frac{1}{2}} \right)^{\wt \left( \mathbf{a} \right)}, & \text{else},
        \end{dcases}
    \end{equation}
    where $\wt: \Fq^t \to \Z$ denotes the Hamming weight of a vector.
    For $r$ \Polocolo{} rounds $\mathcal{R}_r \circ \cdots \circ \mathcal{R}_1$, let $(\mathbf{a}_0, \dots, \mathbf{a}_r) \in \Fq^{n \cdot (r + 1)}$ be a linear trail, i.e.\ $(\mathbf{a}_{i - 1}, -\mathbf{a}_i)$ is the linear approximation of the $i$\textsuperscript{th} round.
    Then, by multiplicativity of the correlation of a linear trail \cite{AC:Beyne21} we have
    \begin{align}
        \abs{\CORR_{\mathcal{R}_r \circ \cdots \circ \mathcal{R}_1} \left( \psi, \mathbf{a}_0, \dots, \mathbf{a}_r \right)}
        &= \prod_{j = 1}^{r}
        \begin{dcases}
            0, & \exists i \! : a_{j - 1, i} \cdot \left( \mathbf{M}^\intercal \mathbf{a}_j \right)_i = 0, \\
            \left( \frac{2^{n + 2}}{q^\frac{1}{2}} \right)^{\wt \left( \mathbf{a}_j \right)}, & \text{else},
        \end{dcases}
        \\
        &\leq \prod_{j = 1}^{r} \left( \frac{2^{n + 2}}{q^\frac{1}{2}} \right)^{\wt \left( \mathbf{a}_j \right)}
        \leq \left( \frac{2^{n + 2}}{q^\frac{1}{2}} \right)^{r}.
    \end{align}
    Since \Polocolo{} requires that $\log_2 \left( q \right) \approx 256$ and the target security level is $\approx 128$ bits, the correlation is too small for a competitive linear distinguisher.

%    \section*{Acknowledgments}
%    The author thanks the reviewers for their valuable comments and suggestions on improving the manuscript.

    \bibliographystyle{alphaurl}
    \bibliography{abbrev3.bib,crypto.bib,literature.bib}

\end{document}